\documentclass[12pt, a4paper]{amsart}
\usepackage{amsfonts,amsmath, amsthm, amssymb, amscd}
\usepackage{latexsym}
\usepackage{tkz-euclide}
\input{amssym.def}
\input{amssym}
\usepackage{hyperref}
\usepackage{aligned-overset}
\usepackage{tikz}
\numberwithin{figure}{section}
\usetikzlibrary{arrows}

\newtheorem{defi}{Definition}[section]
\newtheorem{satz}[defi]{Theorem}
\newtheorem{prop}[defi]{Proposition}

\newtheorem{lemma}[defi]{Lemma}
\newtheorem{bem}[defi]{Remark}
\newtheorem{folg}[defi]{Corollary}

\newcommand {\N}{\mathbb{N}} %% positive integers
\newcommand {\R}{\mathbb{R}} %% reals
\newcommand {\C}{\mathbb{C}} %% complex
\DeclareMathOperator{\vol}{vol}

\DeclareMathOperator{\grad}{grad}

\begin{document}
\title{The shifted Wave equation on non flat harmonic manifolds }

\author{Oliver Brammen}
\date{\today}
\address{Faculty of Mathematics,
Ruhr University Bochum, 44780 Bochum, Germany}
\email{oliver.brammen@rub.de}
\thanks{The author would like to thank Gerhard Knieper and Norbert Peyerimhoff for their support, helpful comments and advise. The author is partially supported by the German Research Foundation (DFG), CRC TRR 191, Symplectic structures in geometry, algebra and dynamics.}

\begin{abstract}
We solve the shifted wave equation 
\begin{align*}
\frac{\partial^2}{\partial t^2}\varphi(x,t)=(\Delta_x+\rho^2)\varphi(x,t)
\end{align*}
on a non compact simply connected harmonic manifold with mean curvature of the horospheres $2\rho>0$. We give an explicit representation of the solution as the inverse dual Abel transform of the spherical means of there initial conditions using the local injectivity of the Abel transform and symmetry properties of the spherical mean value operator. Furthermore we investigate the wave equation using the Fourier transform on harmonic manifolds of rank one. 
Additionally we show an analogous of the classical Paley-Wiener theorem and use it to show an asymptotic Huygens principle as well as asymptotic equidistribution of the energy of a solution of the shifted wave equation under assumptions on the $\mathbf{c}$-function.

\end{abstract}

%%%%%%%%%%%%%%%%%%%%%%%%%%%%%%%%%%%%%%%%%%%%%%%%%%%%%%%%%%%%%%%%%%%%%%

\maketitle

%%%%%%%%%%%%%%%%%%%%%%%%%%%%%%%%%%%%%%%%%%%%%%%%%%%%%%%%%%%%%%%%%%%%%%
\section{Introduction}
 In their paper \cite{Anker2013} the authors solved the shifted wave equation on Damek-Ricci spaces explicitly. These spaces  together with Euclidean and hyperbolic spaces, provide all known examples of non compact simply connected harmonic manifolds. 
 A harmonic manifold is a complete Riemannian manifold $(X,g)$  such that for all $p\in X$ the volume density function in geodesic polar coordinates $\sqrt{g_{ij}(p)}=\theta_q(p)$ only depends on the geodesic distance. The Euclidean and non flat symmetric spaces of rank one are harmonic. It was a long standing conjecture that all harmonic manifolds are of this type, referred to as the Lichnerowicz conjecture \cite{Lich}. The conjecture was proven  for compact simply connected spaces by Szabo\cite{szabo1990} but shortly after this, in 1992  Damek and Ricci \cite{Damek_1992} provided for dimension 7 and higher a class of homogeneous harmonic spaces that are non symmetric. These  manifolds are called Damek-Ricci spaces.  In 2006 Heber \cite{Heber_2006} showed that all homogeneous non compact simply connected harmonic spaces are of the type mentioned above. Since these spaces have a rich algebraic structure one obtains tools from harmonic analysis, see \cite{helgason1994geometric} and \cite{fourierNA2}. In \cite{Biswas2019} the authors showed that one can obtain these tools without the assumption of homogeneity by assuming purely exponential volume growth or equivalently rank one. Furthermore in \cite{PS15} the authors showed that tools like the Abel transform and its dual are accessible without the assumption of rank one. We now use their methods to generalise the results from \cite{Anker2013}. The idea of the proof is identical: We use the symmetries of the mean value operator to express the solution of the shifted wave equation via the inverse dual Abel transform of spherical means of its initial conditions.\\
In section 2 we provide all the generalities on harmonic manifolds needed for this discussion. In section 3 we recall important properties of the Abel transform and its dual form \cite{PS15}, and in section 4 we show the symmetry of the spherical mean operator before solving the wave equation with smooth compactly supported initial conditions explicitly  in section 5.
  In section 6 we investigate the wave equation under the additional assumption that  $X$ is of rank one and thereby obtain a similar results as in \cite{AMBP_2010__17_2_327_0}. 
To conduct this investigation we will use the Fourier transform on $X$. For this purpose we give a brief overview over the Fourier transform on harmonic manifolds of rank one 
and look at the action of the Laplacian under Fourier transform. Then in section 7 we in particular generalise the Paley-Wiener type theorem from  \cite{AMBP_2010__17_2_327_0}  and use it to obtains bounds on the energy of a solution of the shifted wave equation on $X$ under assumptions on the initial conditions. In section 8 we improve the result form the previous section by showing an analogous of the classical Paley-Wiener theorem on harmonic manifolds of rank one, generalising the results from \cite{helgason1994geometric} and \cite{astengocamporesiblasio1997}  for symmetric and non symmetric Damek-Ricci spaces respectively. The main idea of the proof of this theorem is to use the Radon transform from  \cite{Rouvire2021} to translate the problem to the real line. We then use this to obtain an asymptotic Huygens principle (section 9) and asymptotic equidistribution of energy (section 10). Under the assumption that the $\mathbf{c}$-function of $X$ has a polynomial holomorphic extension into a strip on the upper half plane in $\C$ with the first pole of multiplicity one. This generalises the results  of symmetric spaces (\cite{BRANSON1991403},\cite{HELGASON1992279}, \cite{OLAFSSON1992270},\cite{Branson1995},\cite{BRANSON2005429}), non symmetric Damek-Ricci spaces (\cite{AMBP_2010__17_2_327_0}) and gives a non radial version of the results in \cite{AMBP_2005__12_1_147_0}.
%%%%%%%%%%%%%%%%%%%%%%%%%%%%%%%%%%%%%%%%%%%%%%%%%%%%%%%%%%%%%%%%%%%%%%
%%%%%%%%%%%%%%%%%%%%%%%%%%%%%%%%%%%%%%%%%%%%%%%%%%%%%%%%%%%%%%%%%%%%%%
\section{Preliminaries}
In this section we give a brief introduction into non compact simply connected harmonic manifolds.
 For more information 
 we refer the reader to the surveys \cite{kreyssig2010introduction} and \cite{knieper22016}.
 Let $(X,g)$ be a non compact simply connected Riemannian manifold without conjugate points. Denote by $C^k(X)$ the space of $k$-times differentiable functions on $X$ and by $C^k_c(X)\subset C^k(X)$ those with compact support. With the usual conventions for continuous, smooth and analytic functions. Furthermore for $x\in X$ denote by $C^k(X,x)$ the functions in $C^k_c(X)$ radial around $x$ i.e $f\in C^k(X,x)$ if there exists a even function $u\in C_{\text{even}}^k(\R)$ on $\R$ such that $f=u\circ d(x,\cdot)$ where $d:X\times X\to\R_{\geq 0}$ is the distance induced by $g$. Furthermore for $p\geq 1$ $L^p(X)$ refers to the $L^p$-space of $X$ with regards to the measure induced by the metric and integration over a manifold is always be interpreted as integration with respect to the canonical measure on this manifold unless stated otherwise.
   For $p\in X$ and $v\in S_pM$ denote by $c_v :\R\to M$ the unique unit speed geodesic with $c(0)=p$ and $\dot{c}(0)=v$. Define $A_v$ to be the Jacobi tensor along $c_v$ with initial conditions $A_v(0)=0$ and $A^{\prime}(0)=\operatorname{id}$. For details on Jacobi tensors see \cite{KNIEPER2002453}. Then using the transformation formula and the Gauss lemma the volume of the sphere of radius $r$ around $p$ is given by:
\begin{align}\label{eq:volS}
\operatorname{vol}S(p,r)=\int_{S_p M}\operatorname{det} A_v(r)\,dv.
\end{align}

The second fundamental form of $S(p,r)$ is given by $A_v^{\prime}(r)A^{-1}_v(r)$ and the mean curvature by 
\begin{align}\label{eq:meancurv}
\nu_p(r,v)=\operatorname{trace}A_v^{\prime}(r)A_v^{-1}(r).
\end{align}

\begin{defi}
Let $(X,g)$ be a complete non compact simply connected manifold without conjugate points and $SX$ its unit tangent bundle. For $v\in SX$ let $A_v(t)$ be the Jacobi tensor with initial conditions $A_v(0)=0$ and $A^{\prime}_v(0)=\operatorname{id}$. Then $X$ is said to be harmonic if and only if
$$A(r)=\operatorname{det}(A_v(r))\quad\forall v\in SX.$$
Hence the volume growth of a geodesic ball centred at $\pi(v)$ only depends on its radius.
\end{defi}
From  (\ref{eq:meancurv}) one easily concludes that the definition above is equivalent to the mean curvature of geodesic spheres only depending on the radius. More precise the mean curvature of a geodesic sphere $S(x,r)$  of radius $r$ around a point $x\in X$ is given by $\frac{A'(r)}{A(r)}$. 

Using $A_v$ one can construct the Jacobi tensor $S_{v,r}$ along $c_v$  with $S_{v,r}(0)=\operatorname{id}$, $S_{v,r}(r)=0$, and $U_{v,r}=S_{v,-r}$.

Then the stable respectively unstable Jacobi tensor is obtained via the limiting process: 
\begin{align*}
S_v&=\lim\limits_{r\to\infty}S_{v,r}\\
U_v&=\lim\limits_{r\to \infty} U_{v,r}.
\end{align*}
Note that these limits exist \cite{KNIEPER2002453}.

Let $v\in S_p X$ and $c_v$ the unit speed geodesic with initial direction $v$.
Now define for $x\in X$ the Busemann function $b_v(x)=\lim_{t\to \infty}b_{v,t}(x)$, where $b_{t,v}(x)=d(c_v(t),x)-t$.
This limit exists and is a $C^{1,1}$ function on $X$, see for instance \cite{rub.181283719860101}. 
The level sets of the Busemann functions, $H^s_{v}:=b^{-1}_v(s)$ are called horospheres and in the case that $b_v\in C^2(X)$ their second fundamental form in $\pi(v)=p$ is given by $U_v^{\prime}(0)=:U(v)$. Hence their mean curvature is given by the trace of $U(v)$. In the case of a harmonic manifold $v\to \operatorname{trace}U(v)$ is independent of $v\in SX$, hence the mean curvature of horospheres is constant. 
Using this notion of stable and unstable Jacobi tensors Knieper in \cite{knieper2009new} generalised the well known notion of rank for general spaces of nonpositive curvature introduced by Ballmann, Brin and Eberlein \cite{Ballmann1985} 
 to manifolds without conjugated points. 

 Define for $v\in SX$  $S(v):=S'_v(0)$ and $D(v)=U(v)-S(v)$. Then:
\begin{align*}
\mathcal{L}(v)&:=\operatorname{Kern}(D(v))\\
\operatorname{rank}(v)&:=\operatorname{dim}\mathcal{L}(v)+1\\
\operatorname{rank}(X)&:=\operatorname{min}\{\operatorname{rank}(v)\mid v\in SM\}.
\end{align*}
Furthermore Knieper showed that  for a non compact harmonic manifold $\operatorname{rank}(X)=1$ is equivalent to other important notions in geometry these are stated in section 6. 

For $f\in C^2(X)$ the Laplace-Beltrami operator is defined by
\begin{align*}
\Delta f:=\operatorname{div}\operatorname{grad} f
\end{align*}
and for local coordinates $\{x_i\}$ is given by 
\begin{align*}
\Delta f=\sum_{i,j}\frac{1}{\sqrt{\operatorname{det}g}}\frac{\partial}{\partial x_i}\bigl(\sqrt{\operatorname{det}g}g^{ij}\frac{\partial}{\partial x_j}f\bigr)
\end{align*}
where $g=\{g_{ij}\}$ is the matrix which defines the metric tensor $g:TX\times TX\to [0,\infty)$ 
and $\{g^{ij}\}$ its inverse.
 $\Delta$ is by definition linear on $C^{\infty}_c(X)$ and we have 
\begin{align*}
\int_X -\Delta f(x)\cdot f(x) \,dx=\int_X\lVert \nabla f(x)\rVert_g^2 \,dx\quad\forall f\in C^{\infty}_c(X)
\end{align*}
where $\lVert\cdot\rVert_g$ is the norm induced by $g$. 
Hence $-\Delta$ is a non negative symmetric operator. Furthermore $-\Delta$ is formally self adjoint hence by the density of $C^{\infty}_c(X)$ in $L^2(X)$ 
we can extent $\Delta$ to a self adjoint operator on  $L^2(X)$ which in abuse of notation we will again denote by $\Delta$. The above also implies that the spectrum  of $\Delta$ is contained in the negative half line. 
 From now on assume that $(X,g)$ is a non compact simply connected harmonic manifold with mean curvature of the horosphere $h=2\rho$. In this case the authors showed in \cite{BusemannHarmonic} that $\Delta b_v=h$ and hence the Busemann functions as well as all eigenfunctions of $\Delta$ are analytic by elliptic regularity since harmonic manifolds are Einstein, see for instance  \cite[Sec. 6.8]{willmore1996riemannian}, and therefore analytic by the Kazdan-De Truck theorem \cite{ASENS_1981_4_14_3_249_0}. Furthermore the authors in \cite[Corollary 5.2]{PS15} showed that the top of the spectrum of $\Delta$ is given by $-\rho^2$.
 \begin{lemma}[\cite{Biswas2019}, Lemma 3.1]\label{lemma:Lapace}
Let $f$ be a $C^2$ function on $(X,g)$  and $u$ a $C^{\infty}$ function on $\R$. Then we have:
$$\Delta (u\circ f)=(u''\circ f)\lVert \operatorname{grad} f\rVert_g^2 +(u'\circ f)\Delta f.$$
where $\lVert\cdot\rVert_g^2=g(\cdot,\cdot)$.
\end{lemma}
 With Lemma \ref{lemma:Lapace} we can calculate the spherical and horospherical part of the Laplacian, by
choosing $f=d_x$ for some $x\in X$. We obtain with $\Delta d_x(r)=\frac{A^{\prime}(r)}{A(r)}\circ d_x(r)$ using spherical coordinates around $x$
\begin{align}\label{equ:radiallaplace}
\Delta(u\circ d_x)=u^{\prime\prime}\circ d_x +u^{\prime}\circ d_x\cdot \frac{A^{\prime}}{A}\circ d_x.
\end{align}
For the Busemann function  $f=b_v$ with $\Delta b_v=h=2\rho$ we obtain using horospherical coordinates 
\begin{align}\label{eq:horolaplace}
\Delta( u\circ b_v)=u^{\prime\prime}\circ b_v+h\cdot u^{\prime}\circ b_v.
\end{align}
From this we have that the radial part of the Laplacian, does only depend on the radius and not on specific points. Therefore we obtain:
\begin{lemma}\label{equivalenz}
Let $f:X\to \C$ be a $C^{\infty}_c(X)$ function and $x\in X$ then for the mean value operators
$$M_x f (r):=\frac{1}{\vol(S(x,r))}\int_{S(x,r)}f (z)\,dz$$
and
$$R_x(f)(y):=M_xf(d(x,y))$$

we have $$\Delta R_xf(y)=R_x(\Delta f)(y).$$
Especially we have for  $$L_{A}:=\frac{d^2}{dr^2}+\frac{A^{\prime}(r)}{A(r)}\frac{d}{dr}$$ 
that
$$L_A M_x(f)(r)=M_x(\Delta f)(r).$$
\end{lemma}
\begin{proof}
We can decompose the Laplacian 
$$\Delta f (y)=\Delta_{S(x,d(x,y))}f(y)+\Delta_{\text{radial}}f(y).$$
Where $\Delta_{S(x,d(x,y))}$ denotes the Laplacian of $S(x,d(x,y))$ and $\Delta_{\text{radial}}$ is defined by: 
$$(\Delta_{\text{radial}}f)(c_v(r))=L_A(f\circ c_v)(r),$$
where for $v\in SX$, $c_v$ is the geodesic corresponding to the initial conditions $c_v(0)=\pi(v)$ and $\dot{c}_v(0)=v$.
Since $S(x,d(x,y))$ is closed  Greens first identity implies:
$$\int_{S(x,d(x,y))}\Delta_{S(x,d(x,y))}f(z)\,dz=0.$$
Now the radial  part of the Laplacian only depends on radial derivatives and the mean curvature of the geodesic sphere which since $X$ is harmonic also only depends on the radius hence:
\begin{align*}
R_x(\Delta f)(y)&=R_x(\Delta_{\text{radial}}f)(y)\\
\overset{\text{ $X$ is harmonic}}&{=}\Delta_{\text{radial}}R_x(f)(y)\\
&=\Delta R_x(f)(y).
\end{align*}
The second part of the Lemma follows now from (\ref{equ:radiallaplace}).
\end{proof}
\begin{bem}
Note that the fact that the Laplace operator commutes with the mean value operator is equivalent to $X$ being harmonic. See for instance \cite[Lemma 1.1]{szabo1990}.
\end{bem}
\begin{lemma}\label{lemma:radself}
Let $x_0\in X$ then $R_{x_0}:C_c^{\infty}(X)\to C_c^{\infty}(X,x_0)$ is self adjoint with respect to the $L^2$-product on $X$ i.e.:
\begin{align*}
\int_X(R_{x_0}f)(x)g(x)\,dx=\int_Xf(x)(R_{x_0}g)(x)\,dx \quad \forall f,g\in C^{\infty}_c(X).
\end{align*}
\end{lemma}
\begin{proof}
Let $f,g\in C^{\infty}_c(X)$ and $x_0\in X$.
We integrate in geodesic polar coordinates using equation (\ref{eq:volS}) and the fact that $X$ is harmonic:
\begin{align*}
\int_X(R_{x_0}f)(x)g(x)\,dx&=\frac{1}{\omega_{n-1}}\int_0^{\infty}\Big( \int_{S_{x_0}X}f(\exp(rv))\,dv\\
&\cdot  \int_{S_{x_0}X}g(\exp(rv))\,dv\Big )\,A(r)\,dr\\
&=\int_Xf(x)(R_{x_0}g)(x)\,dx 
\end{align*}
where $\omega_{n-1}=\operatorname{vol} S^{n-1}$. 
\end{proof}

\section{The Abel Transform and Its Dual}

Peyerimhoff and Samion discussed the Abel transform and its dual for radial functions as well as its connection to the radial Fourier transform in \cite{PS15}. We will use these to construct a solution to the shifted wave equation. Therefore we recall the definition and some imported facts that we will need in the prove of our main theorems. 
 For this purpose we need the following version of the Co-area formula.

 \begin{satz}[{\cite[p.160]{chavel2006}}]
 Let $M$ be a connected Riemannian manifold. Given a $C^1$-function $f:M\to\R$ such the gradient $\operatorname{grad}f$ never vanishes on $M$, let $S_t$ denote the hypersurface defined by $S_t=\{x\in M\mid f(x)=t\}$, $t\in \R$. Then, for any $g\in C^0_c(M)$,
 \begin{align*}
 \int_M g(x) \,dx=\int_{\R} \int_{S_t}\frac{g(y)}{\lVert \operatorname{grad} f(y)\rVert_g}\,dy \,dt.
 \end{align*}
 \end{satz}
 
Let $x_0\in X$ and $v\in S_{x_0}X$ then  $H_{v}^s=b_v^{-1}(s)$ denote the horospheres and $N(x)=-\grad b_v(x)$.
Then the  map  
\begin{align*}
\Psi_{v,s}:H_v^0\to H_v^s\\
 x\mapsto \exp(-sN(x))
 \end{align*}
 is a diffeomorphism
  and
 \begin{align}\label{eq:diffeo1}
  \Psi_{v}:\R\times H_{v}^0\to X\\
  \Psi_{v}(s,x)=\Psi_{v,s}(x)\nonumber
  \end{align}
is an orientation preserving diffeomorphism. Furthermore the Jacobian of $\Psi_{v,s}$ is given by $e^{hs}$ (see \cite[Proposition 3.1]{PS15}).
Hence, for a measurable function $f:X\to \R$  we get  : 
\begin{align}\label{eq:diffeo2}
 \int_{H_{v}^s}f(z)\,dz=e^{sh}\int_{H_{v}^0}f(\Psi_s(z))\,dz.
 \end{align}
\begin{defi}
For $v\in S_{x_0}X$ and define 
$$j:C^{\infty}_{\text{even}}(\R)\to C^{\infty}(X)$$
$$(jf)(x)=e^{-\rho b_v(x)}f(b_v(x))$$
and $$a:C^{\infty}_{\text{even}}(\R)\to C^{\infty}(X,x_0)$$
by $$a(f)(y)=M_{x_0}( j(f))\circ d(x_0,y).$$

The dual with respect to the $L^2$-inner product of $\R$ and $X$ is called the Abel transform and is denoted by $\mathcal{A}$. This means that for every $g\in C^{\infty}(X,x_0)$ and $f\in C^{\infty}_{\text{even}}(\R)$  we have
\begin{align*}
\int_\R \mathcal{A}(g)(s)f(s)\,ds=\int_X g(x)a(f)(x)\,dx.
\end{align*}
\end{defi}
Furthermore the authors in \cite{PS15} showed in Proposition 3.5  that:
 \begin{lemma}
For $f\in C^{\infty}_c(X,x_0)$  we have: 
\begin{align*}
\mathcal{A}(f)(s)&=e^{-\rho s}\int_{H^s_v}f(z)\,dz\\
&=e^{\rho s}\int_{H^0_v}f(\Psi_{v,s}(z))\,dz.
\end{align*}
Furthermore $\mathcal{A}(f)$ is smooth, has compact support and is even. 
\end{lemma}
\begin{proof}
Let $f\in C^{\infty}_{c}(X,x_0)$ and define 
\begin{align*}
g(s):=e^{-\rho s}\int_{H^s_v}f(z)\,dz.
\end{align*}
Then bottom equality follows immediately from (\ref{eq:diffeo2}). Therefore we only need to show that:
\begin{align}\label{eq:lemmaabel}
\int_{\R} g(s)h(s)\,ds=\int_X f(x)a(h)(x)\,dx \quad \forall h\in C^{\infty}_{\text{even}}(\R)
\end{align}
and that $g(s)$ is even, since the smoothness follows after showing the equality from the smoothness of $\Psi_{s,v}$ in $s$.
Now we prove (\ref{eq:lemmaabel})
\begin{align*}
\int_{\R} g(s)h(s)\,ds&=\int_{\R} h(s)e^{-\rho s}\int_{H^s_v}f(z)\,dz\,ds\\
&=\int_{\R}\int_{H^s_v} h(b_v(z))e^{-\rho s}f(z)\,dz\,ds\\
\overset{\text{Co-area formula}}&{=}\int_Xf(x)e^{-\rho b_v(x)}h(b_v(x))\,dx\\
&=\int_Xf(x)j(h)(x)\,dx\\
&=\int_XR_{x_0}(f)(x)j(h)(x)\,dx\\
\overset{\text{Lemma \ref{lemma:radself}}}&{=}\int_Xf(x)R_{x_0}(j(h))(x)\,dx\\
&=\int_Xf(x)a(h)(x)\,dx.
\end{align*}
Let for $\lambda\in \C$, $\varphi_{\lambda,x_0}$ be a eigenfunction  of the Laplacian with eigenvalue $-(\lambda^2+\rho^2)$  radial around $x_0$ with $\varphi_{\lambda,x_0}(x_0)=1$.
Now evenness follows similar  to (\ref{eq:lemmaabel}) if we observe that since the Laplacian commutes  with $R_{x_0}$ and by (\ref{eq:horolaplace})
$e^{(i\lambda -\rho)b_v(x)}$ is for all $\lambda\in \C$ a eigenfunction of $\Delta$ with eigenvalue $-(\lambda^2+\rho^2)$ we have
\begin{align}\label{eq:buseef}
R_{x_0}\left(e^{(i\lambda -\rho)b_v(\cdot)}\right)(x)=\varphi_{\lambda,x_0}(x).
\end{align}
Then using this and integration in horospherical coordinates yields: 
\begin{align*}
\int_{\R}g(s)e^{i\lambda s}\,ds&=\int_{\R} e^{i\lambda s} e^{-\rho s}\int_{H^s_v}f(z)\,dz\,ds\\
&=\int_{\R}\int_{H^s_v} e^{i\lambda b_v(z)}e^{-\rho s}f(z)\,dz\,ds\\
\overset{\text{horospherical coordinates}}&{=}\int_Xf(x)e^{(i\lambda-\rho)b_v(x)}\,dx\\
\overset{\text{f radial +Lemma \ref{lemma:radself}}}&{=}\int_Xf(x)R_{x_0}(e^{(i\lambda-\rho)b_v(\cdot)})(x)\,dx\\
\overset{\text{(\ref{eq:buseef})}}&{=}\int_X f(x)\varphi_{\lambda,x_0}(x)\,dx.
\end{align*}
Now we have that $\varphi_{\lambda,x_0}=\varphi_{-\lambda,x_0}$, hence:
\begin{align*}
\int_{\R}g(s)e^{i\lambda s}\,ds=\int_{\R}g(s)e^{-i\lambda s}\,ds.
\end{align*}
This in tune implies that:
\begin{align*}
\int_{\R}e^{i\lambda s} (g(s)-g(-s))\,ds=0 \quad \forall \lambda \in \C.
\end{align*}
 By taking $\lambda\in \R$ this implies that $g$ is even.

\end{proof}

Furthermore the authors showed in \cite[Proposition 3.10]{PS15} that the Euclidean Fourier transform of the Abel transform is equal to the radial Fourier transform, given  for a function radial around $x_0$ with compact support by 
$$\hat{f}^{x_0}(\lambda)=\int_Xf(x)\varphi_{\lambda,x_0}(x)\,dx,$$
where $\varphi_{\lambda,x_0}$ is the radial eigenfunction of the Laplacian around $x_0$ with eigenvalue $-(\lambda^2+\rho^2)$ and $\varphi_{\lambda,x_0}(x_0)=1$. 
This means that 
\begin{align}\label{eq:abel-fourier}
\hat{f}^{x_0}(\lambda)=\mathcal{F}(\mathcal{A}(f))(\lambda)
\end{align}
where $\mathcal{F}(u)(\lambda)=\int_{\R}e^{i\lambda s}u(s)\,ds$ for $u:\R\to\R$ sufficiently regular is the Euclidian Fourier transform.
\begin{bem}
Applying $\mathcal{F}^{-1}$ to both sides in equation (\ref{eq:abel-fourier}) yields that the Abel transform and thereby its dual are independent of the choice of $v\in S_{x_0}X$. See also Lemma \ref{lemma:PW2}.
\end{bem}
\begin{satz}[\cite{PS15}, Theorem 3.8]
The dual Abel transform is a topological isomorphism between the spaces of smooth even functions on $\R$ and smooth radial functions around $x_0$.
\end{satz}
This fact is going to be exploited to  characterise solutions of 
 the wave equation on $X$ with smooth initial conditions with compact support. 
\section{Symmetry of the Mean Value Operator}
From here one out we will consider complex valued functions $u:X\to\C$, where the Laplacian of $u$ is given via the decomposition of $u$ in real and imaginary part $u=u_1+iu_2$ by $\Delta u= \Delta u_1 +i\Delta u_2$. 
The proof of the following lemma follows the lines of the proof of Theorem 17 in \cite{Helgason1959} which in turn follows the proof in \cite[p.334]{MR1513094}.
\begin{lemma}\label{lemma:meanvalue}
Let $(X,g)$ be a non compact simply connected harmonic manifold, and $u:X\times X\to\C$ a twice continuous differentiable function with
$$\Delta_{1}u(x,y)=\Delta_{2} u(x,y)\quad\forall x,y\in X,$$
where $\Delta_i$ denotes to Laplacian with respect to the $i$-th variable. 
Then for each $(x_0,y_0)\in X\times X$ we have 
\begin{align*}
&\frac{1}{\vol(S(x_0,r))}\frac{1}{\vol(S(y_0),s)}
 \int_{S(x_0,r)}\int_{S(y_0,s)}u(z_1,z_2)\,dz_2\,dz_1
 \\&=\frac{1}{\vol(S(x_0,s))}\frac{1}{\vol(S(y_0,r))}
 \int_{S(x_0,s)}\int_{S(y_0,r)}u(z_1,z_2)\,dz_2\,dz_1
\end{align*}

for all $r,s\geq 0$.
\end{lemma}
\begin{proof}

Let $(x_0,y_0)\in X\times X$ be arbitrary points define 
\begin{align*}
U(x,y):=&\frac{1}{\vol(S(x_0,r))}\frac{1}{\vol(S(y_0,s))}
 &\int_{S(x_0,r)}\int_{S(y_0,s)}u(z_1,z_2)\,dz_2\,dz_1
\end{align*}
with $r=d(x_0,x)$ and $s=d(y_0,y)$. Then $U$ can both be viewed as a function on $X\times X$ and $\R^+\times\R^+$.

 Since the Laplacian $\Delta$ commutes with the mean value operator (see Lemma \ref{equivalenz}) and $u$ is twice continuous differentiable we have:
 \begin{align*}
\Delta_1U(x,y)&=\Delta_1R_{x_0}\big((z,y)\to R_{y_0}(u(z,\cdot))(y)\big)(x)\\
&=R_{x_0}\big((z,y)\to \Delta_1R_{y_0}(u(z,\cdot))(y)\big)(x)\\
&=R_{x_0}\big((z,y)\to R_{y_0}(\Delta_1u(z,\cdot))(y)\big)(x)\\
&=R_{x_0}\big((z,y)\to R_{y_0}(\Delta_2u(z,\cdot))(y)\big)(x)\\
&=R_{x_0}\big((z,y)\to \Delta_2R_{y_0}(u(z,\cdot))(y)\big)(x)\\
&=\Delta_2R_{x_0}\big((z,y)\to R_{y_0}(u(z,\cdot))(y)\big)(x)\\
&=\Delta_2U(x,y).
 \end{align*}
 Then with the representation of the Laplacian in radial coordinates (see(\ref{equ:radiallaplace})) we have: 
\begin{align*}
\frac{\partial^2 U}{\partial r^2}+\frac{A'(r)}{A(r)} \frac{\partial U}{\partial r}=\frac{\partial^2 U}{\partial s^2}+\frac{A'(s)}{A(s)} \frac{\partial U}{\partial s}.
\end{align*}
If we set $F(r,s)=U(r,s)-U(s,r)$ we obtain:
\begin{align}\label{eq:one}
\frac{\partial^2 F}{\partial r^2}+\frac{A'(r)}{A(r)} \frac{\partial F}{\partial r}&-\Big(\frac{\partial^2 F}{\partial s^2}+\frac{A'(s)}{A(s)} \frac{\partial F}{\partial s}\Big)=0,\\
\label{eq:asym}F(r,s)=&-F(s,r).
\end{align}
Our goal is it now to show that $F\equiv 0$. Since $F(r,r)=0$ is sufficient to show that all partial derivatives of $F$ vanish.
 We have:
\begin{align*}
 A'(r)\frac{\partial F}{\partial r}\frac{\partial F}{\partial s}=& \frac{\partial}{\partial r}\Big( A(r)\frac{\partial F}{\partial r}\frac{\partial F}{\partial s}\Big)-A(r)\frac{\partial^2 F}{\partial^2 r}\frac{\partial F}{\partial s}\\
 &-A(r)\frac{\partial F}{\partial r}\frac{\partial^2 F}{\partial s\partial r},
\end{align*}
and
$$\frac{\partial}{\partial s}\big(\frac{\partial F}{\partial r}\big )^2=2\frac{\partial F}{\partial r}\frac{\partial^2 F}{\partial s\partial r}, \quad\frac{\partial}{\partial s}\big(\frac{\partial F}{\partial s}\big )^2=2\frac{\partial F}{\partial s}\frac{\partial^2 F}{\partial s^2}.$$
 Therefore multiplying (\ref{eq:one}) by $2A(r)\frac{\partial F}{\partial s}$ we obtain:
\begin{align}\label{eq:tri1}
-A(r)\frac{\partial}{\partial s}\Big ( \big (\frac{\partial F}{\partial r} \big)^2+\big (\frac{\partial F}{\partial s}\big)^2\Big)+2\frac{\partial}{\partial r}\Big (A(r)\frac{\partial F}{\partial r}\frac{\partial F}{\partial s}\Big)\\
-2\frac{A'(s)A(r)}{A(s)}\Big(\frac{\partial F}{\partial s}\Big )^2=0.\nonumber
\end{align}
Now set $$L_1:=A(r)\Big ( \big (\frac{\partial F}{\partial r} \big)^2+ \big (\frac{\partial F}{\partial s} \big)^2\Big)$$
and
$$L_2:= 2\Big (A(r)\frac{\partial F}{\partial r}\frac{\partial F}{\partial s}\Big).$$
 Let $C>0$ be arbitrary and consider the line $r+s=C$. We want to integrate the formula (\ref{eq:tri1}) over the triangle $D$ with oriented boundary $\partial D=OMN$  (see Figure \ref{fig:1}), where $O=(0,0)$, $M=(\frac{C}{2},\frac{C}{2})$ and $N=(0,C)$, using Stokes theorem. With this we then show $F$ vanishes on $D$.
For this we first need the check that the expressions in (\ref{eq:tri1}) have no singularities in $D$. The critical term is $2\frac{A'(s)A(r)}{A(s)}.$
To rule out such a singularity let $r\leq s$ then since $A$ is monotonous increasing we have  $\frac{A'(s)A(r)}{A(s)}\leq A'(s)$ and $A'(0)=1$ hence we have no singularity at $O$. 
Using  Stokes theorem and equation (\ref{eq:tri1})  we get:
\begin{align}
\label{eq:tri0}\iint_D \frac{2A(r)A'(s)}{A(s)}\Big(\frac{\partial F}{\partial s}\Big )^2\,dr\,ds&=\iint_D\frac{\partial L_2}{\partial r}-\frac{\partial L_1}{\partial s}\,dr\wedge ds\\
&=\int_Dd(L_1dr+L_2ds)\nonumber\\ 
&=\int_{\partial D} L_1dr+L_2 ds.\nonumber
\end{align}
\begin{figure}[ht]
\centering
\caption{The triangle $D$ with oriented boundary $\partial D=OMN$.}
\label{fig:1}
\begin{tikzpicture}

  \draw[->] (-1,0) -- (5,0);
  \draw[->] (0,-1) -- (0,5);
  \draw (0,3)-- (3,0);
  \draw (1.5,1.5)-- (0,0);

    \draw[-{Stealth[length=3mm]}] (1.5,1.5) -- (0.75,2.25);
   \draw[-{Stealth[length=3mm]}] (0,0) -- (0.75,0.75);
     \draw[-{Stealth[length=3mm]}] (0,3) -- (0,1.5);
  \draw[fill=black](0,0)circle(1pt);
  \draw[fill=black](0,3)circle(1pt);
    \draw[fill=black](3,0)circle(1pt);
  \draw[fill=black](1.5,1.5)circle(1pt);
  \node [left] at (0,-0.2) {$O$};
  \node [left] at (0,3) {$N$};
  \node [right] at (1.5,1.5) {$M$};
  \node [below] at (5,0) {$r$};
   \node [left] at (0,5) {$s$};
   \node [right] at (0.35,1.5) {$D$};

  \end{tikzpicture}
\end{figure}
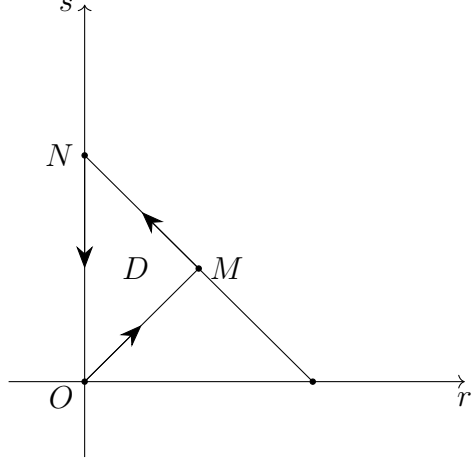%\\

We have to break the path along the boundary into the three lines. First consider the line $r=s$ parameterised by the curve $\gamma_1(t)=(t,t)$ ending at $M$ denoted by $OM$. Then we have $\dot{\gamma}_1=(1,1)$ and therefore: 
\begin{align}\label{eq:tri2}
\int_{OM} L_1dr+L_2 ds=&\int_0^{C/2}A(t)\Biggl ( \big (\frac{\partial F}{\partial r}(t,t) \big)^2 +\big (\frac{\partial F}{\partial s}(t,t) \big)^2\\
&+2\big (\frac{\partial F}{\partial r}(t,t)\frac{\partial F}{\partial s}(t,t)\big )\Biggr )\,dt.\nonumber
\end{align}
Since $F(\gamma_1(t))=F(t,t)=0$ for all $t\geq 0$ we have 
\begin{align}\label{eq:1}
0=DF(\gamma_1(t))\cdot \dot{\gamma_1}(t)=\frac{\partial F}{\partial r}(t,t) +\frac{\partial F}{\partial s}(t,t)\quad \forall t\geq 0,
\end{align}
hence
\begin{align*}
\Big( \frac{\partial F}{\partial r} (t,t)+ \frac{\partial F}{\partial s}(t,t)\Big )^2=0.
\end{align*}
From this we conclude that the integral (\ref{eq:tri2}) vanishes.\\
Next we consider the line $ON$. We have that $A(r)=0$ therefore $L_1=0=L_2$ on $ON$ and 
$$\int_{ON} L_1dr+L_2 ds=0.$$

Lastly we consider the curve jointing $N$ and $M$ given by $\gamma_2(t)=(t,C-t)$. Then we have $\dot{\gamma}_2(t)=(1,-1)$ and obtain:
\begin{align*}
\int_{MN} L_1dr+L_2 ds&= \int_{C/2}^0 2\Big (A(t)\frac{\partial F}{\partial r}(t,C-t)\frac{\partial F}{\partial s}(t,C-t)\Big)\\
&-A(t)\Big ( \big (\frac{\partial F}{\partial r}(t,C-t) \big)^2 +\big (\frac{\partial F}{\partial s}(t,C-t) \big)^2\Big)\,dt\\
&= \int_0^{C/2}  A(t)\Big (\frac{\partial F}{\partial r}(t,C-t)-\frac{\partial F}{\partial s}(t,C-t)\Big )^2\,dt.
\end{align*}
Now we have using (\ref{eq:tri0})
\begin{align*}
 \int_0^{C/2} & A(t)\Big (\frac{\partial F}{\partial r}(t,C-t)-\frac{\partial F}{\partial s}(t,C-t)\Big )^2\,dt\\
&+\iint_D \frac{2A(r)A'(s)}{A(s)}\big(\frac{\partial F}{\partial r}\big )^2\,dr\,ds
=0.
\end{align*}
Since $A'(s)\geq 0$ both integrals are non negative. This implies that 
\begin{align}\label{eq:2}
0=\frac{\partial F}{\partial r}(t,C-t)-\frac{\partial F}{\partial s}(t,C-t)=DF(\gamma_2(t))\cdot\dot{\gamma}_2(t)\quad \forall t\geq 0.
\end{align}
Now since $C>0$ is arbitrary $(\ref{eq:1})$ together with $(\ref{eq:2})$ implies that all partial derivatives of $F$ vanish and therefore that $F$ is constant on the left side of the line $(t,t)$. Since $F(r,r)=0$ we conclude $F(s,r)=0$ on the left side of the line $(t,t)$. Since $F$ is antisymmetric, see equation (\ref{eq:asym}), the same holds true for the the rest of $\R^2_+$ hence the claim follows. 

\end{proof}
\begin{folg}\label{folg:meanvalue}
Under the conditions and with the notations of the proof of Lemma \ref{lemma:meanvalue} we have that $U(r,0)=U(0,r)$ for all $r\geq 0$ hence we obtain:
\begin{align}\label{eq:conclusion}
 M_{y_0}(u(x_0,\cdot))(r)=M_{x_0}(u(\cdot,y_0))(r).
 \end{align}
\end{folg}
With a classical Lemma by Willmore \cite[p.249]{willmore1996riemannian} one can deduce a near equivalence in Corollary \ref{folg:meanvalue}. 

\begin{folg}
Let $u:X\times X\to \R$ be a smooth function such that equation (\ref{eq:conclusion}) holds for a small neighbourhood of  $(x_0,y_0)\in X\times X$ and all small $r>0$ then:
$$\Delta_1u(x_0,y_0)=\Delta_2u(x_0,y_0).$$
\end{folg}
\begin{proof}
We have by \cite[p.249]{willmore1996riemannian} for $f\in C^{\infty}(X)$, $x\in X$ and $r>0$:
\begin{align*} 
 M_x(f)(r)=f(x)+\frac{1}{2n}\Delta f(x) r^2+ O(r^4)\quad\text{ for } r\to 0,
 \end{align*}
 where $n=\operatorname{dim} X$.
Applying this to $u$ yields:
\begin{align*}
 M_{x_0}(u(\cdot,y_0))(r)=u(x_0,y_0)+\frac{1}{2n}\Delta_1 u(x_0,y_0) r^2+ O(r^4)\quad\text{ for } r\to 0,\\
 M_{y_0}(u(x_0,\cdot))(r)=u(x_0,y_0)+\frac{1}{2n}\Delta_2 u(x_0,y_0)  r^2+ O(r^4)\quad\text{ for } r\to 0,
\end{align*}
Since the terms on the left hand side coincide, we obtain the claim. 
\end{proof}

\section{The Shifted Wave Equation}
In this section we solve the shifted wave equation:
\begin{align*}
\varphi:X\times \R&\to\C\\
\frac{\partial^2}{\partial t^2}\varphi(x,t)&=(\Delta_x+\rho^2)\varphi(x,t)
\end{align*}
on $X$ with initial conditions $$\varphi(x,0)=f(x)\in C_c^{\infty}(X)$$ and $$\left.\frac{\partial}{\partial t}\right\vert_{t=0} \varphi(x,t)=g(x)\in C_c^{\infty}(X),$$ via the inverse Abel transform. This is analogous to \'{A}sgeirsson characterisation of the solutions of the wave equation on $\R^n$ \cite{MR1513094} and generalises work on non compact symmetric spaces and Damek-Ricci spaces by \cite{Helgason1959}, \cite{MR1923488} and  \cite{Anker2013} respectively. The methods used are to a large part identical and rely heavily on \cite[Theorem 3.8]{PS15} and Corollary \ref{folg:meanvalue}. Where our approach differs is in that we do not have an explicit formula for the inverse dual Abel transform and hence need to rely on the local infectivity of the dual Abel transform shown in \cite[Theorem 3.8]{PS15} to obtain the existence of solutions and that they posses finite speed of propagation.
\begin{lemma}\label{lemma:horo}
Let $x_0\in X$,  $v\in S_{x_0}X$
and $u:X\times \R \to \C$ be a $C^2(X\times \R)$ function. Then for the function $U:X\times X\to \C$ defined by by $U(x,y)=e^{-\rho b_v(y)}u(x,b_v(y))$ the Laplacian $\Delta_2$ of $U$ with respect to the second variable is given by
\begin{align*}
\Delta_2 U(x,y)=e^{-\rho b_v(y)} (\frac{\partial^2}{\partial t^2}-\rho^2) u(x,\cdot))\circ b_v(y).
\end{align*}
\end{lemma}
\begin{proof}
Define $h:X\times\R\to \C$ by $h(x,t)=e^{-\rho t}u(x,t)$, then 
by the representation of the Laplacian in horospherical coordinates  (\ref{eq:horolaplace}) the  Laplacian with respect to the second variable can be expressed by 
\begin{align}\label{eq:horo1}
\Delta_2 U(x,y)=\big(\frac{\partial^2}{\partial t^2}h(x,\cdot)+2\rho \frac{\partial}{\partial t}h(x,\cdot)\big)\circ b_v(y).
\end{align}
With
\begin{align*}
\frac{\partial}{\partial t}h(x,t)&=-\rho e^{-\rho t} u(x,t)+e^{-\rho t}\frac{\partial}{\partial t} u(x,t),\\
\frac{\partial^2}{\partial t^2}h(x,t)&=\rho^2  e^{-\rho t} u(x,t) -2\rho  e^{-\rho t}\frac{\partial}{\partial t} u(x,t)+e^{-\rho t}\frac{\partial^2}{\partial t^2} u(x,t).
\end{align*}
We get:
\begin{align}
\frac{\partial^2}{\partial t^2}h(x,t)+2\rho \frac{\partial}{\partial t}h(x,t)&=\rho^2  e^{-\rho t} u(x,t) -2\rho  e^{-\rho t}\frac{\partial}{\partial t} u(x,t)\nonumber\\
&+e^{-\rho t}\frac{\partial^2}{\partial t^2} u(x,t)-2\rho^2 e^{-\rho t} u(x,t)\nonumber\\
&+2\rho e^{-\rho t}\frac{\partial}{\partial t} u(x,t)\nonumber\\
&=e^{-\rho t}\big(\frac{\partial^2}{\partial t^2} u(x,t)-\rho^2 u(x,t)\big)\nonumber\\
&=e^{-\rho t} (\frac{\partial^2}{\partial t^2}-\rho^2) u(x,t).\label{eq:horo2}
\end{align}
Now plugging (\ref{eq:horo2}) into (\ref{eq:horo1}) yields the claim. 
\end{proof}

\begin{satz}\label{thm:wave1}
Let $\varphi:X\times \R\to\C$ be a $C^{\infty}$ solution of the shifted wave equation 
\begin{align*}
\frac{\partial^2}{\partial t^2}\varphi(x,t)=(\Delta_x+\rho^2)\varphi(x,t)
\end{align*}
on $X$ with initial conditions $\varphi(x,0)=f(x)\in C_c^{\infty}(X)$ and $$\left.\frac{\partial}{\partial t}\right\vert_{t=0} \varphi(x,t)=g(x)\in C_c^{\infty}(X)$$
then
$$\varphi(x,t)=(a)^{-1}((M_xf)\circ d(x_0,\cdot))(\lvert t\rvert )+\int_0^{\lvert t \rvert}(a)^{-1}((M_xg)\circ d(x_0,\cdot))(s)\,ds,$$
where $a$ is the dual Abel transform on $X$ based at a point $x_0\in X$. 
\end{satz}
\begin{proof}
Let $x_0\in X$ and  $v\in S_{x_0}X$. And denote by $\Delta_i$ the Laplacian with respect to the $i$-th variable.
First consider a solution to the wave equation $\varphi_1(x,t)$ with initial conditions $\varphi_1(x,0)=f(x)$ and $\frac{\partial}{\partial t} \varphi_1(x,0)=0$ for all $x\in X$. Because of this we can assume that $\varphi_1$ is even in $t$. Define the function $$\Phi_1:X\times\ X \to \C$$ by $$\Phi_1(x,y):=e^{-\rho b_v(y)} \varphi_1(x,b_v(y)).$$
Then since $\varphi_1(x,t)$  is a solution of the wave equation we have:
\begin{align*}
\Delta_1 \Phi_1(x,y) &=e^{-\rho b_v(y)} \Delta_1 \varphi_1(x,b_v(y))\\
&=e^{-\rho b_v(y)}\Big( \big((\frac{\partial^2}{\partial t^2}-\rho^2) \varphi_1(x,\cdot)\big)\circ b_v(y)\Big ).
\end{align*}
 Furthermore by Lemma \ref{lemma:horo} we have that:
 \begin{align*}
 \Delta_2 \Phi_1(x,y)&=e^{-\rho b_v(y)}\Big( \big(\frac{\partial^2}{\partial t^2}-\rho^2) \varphi_1(x,\cdot)\big)\circ b_v(y)\Big).
 \end{align*}
 Therefore:
$$\Delta_1\Phi_1=\Delta_2 \Phi_1.$$
 Now we can apply  Corollary \ref{folg:meanvalue}  above and  obtain that for every pair $x,y \in X$ 
\begin{align*}
a(t\mapsto \varphi_1(x,t))(y)&=M_{x_0}(e^{-\rho b_v(\cdot)}\varphi_1(x,b_v(\cdot)))\circ d(x_0,y)\\%\circ d(x_0,y)
&=M_{x_0}(\Phi_1(x,\cdot))\circ d(x_0,y)\\
&=M_{x}(\Phi_1(\cdot,x_0))\circ d(x_0,y)\\
&=M_{x}(e^{-\rho b_v(x_0)}\varphi_1(\cdot,b_v(x_0))\circ d(x_0,y)\\
&=M_x(f)\circ d(x_0,y),
\end{align*}
 where $a:C^{\infty}_{\text{even}}(\R)\to C^{\infty}(X,x_0)$ denotes the dual Abel transform with the choice of $v\in S_{x_0}X$ as above.
Hence by Theorem 3.8 in  \cite{PS15} we get for every $t \in \R$ and $x\in X$:
$$ \varphi_1(x,t)=a^{-1}\big(M_x(f)\circ d(x_0,\cdot)\big)(\lvert t\vert ).$$
Now let $\varphi_2$ be a solution of the wave equation on $X$ with $\varphi_2(x,0)=0$ and $\frac{\partial}{\partial t} \varphi_2(x,0)=g(x)$ for all $x\in X$. 
Then the initial conditions imply: 
$$ \frac{\partial^2}{\partial t^2}\varphi_2(x,0)=(\Delta +\rho^2)\varphi_2(x,0)=0,$$
hence we can assume that $\frac{\partial}{\partial t} \varphi_2(x,t)$ is for all $x\in X$ a smooth even function in $t$. 
Define
 $$\Phi_2(x,y):=e^{-\rho b_v(y)}\frac{\partial}{\partial t} \varphi_2(x,b_v(y)).$$
 Since $\varphi_2$ is a solution of the wave equation 
 \begin{align*}
 \Delta_1 \Phi_2(x,y)=e^{-\rho b_v(y)}\Big( \big(\frac{\partial^2}{\partial t^2}-\rho^2) \frac{\partial}{\partial t}\varphi_2(x,\cdot)\big)\circ b_v(y)\Big)
 \end{align*}
 and by Lemma \ref{lemma:horo} 
 \begin{align*}
 \Delta_2 \Phi_2(x,y)=e^{-\rho b_v(y)} \Big(\big(\frac{\partial^2}{\partial t^2}-\rho^2) \frac{\partial}{\partial t}\varphi_2(x,\cdot)\big)\circ b_v(y)\Big).
 \end{align*}
 Hence
 $$\Delta_1\Phi_2=\Delta_2  \Phi_2.$$
 Now we can again apply Corollary \ref{folg:meanvalue} and  obtain that for every pair $x,y \in X$ 

\begin{align*}
a(t\mapsto \frac{\partial}{\partial t}\varphi_2(x,t))(y)&=M_{x_0}(e^{-\rho b_v(\cdot)}\frac{\partial}{\partial t}\varphi_2(x,b_v(\cdot)))\circ d(x_0,y)\\
&=M_{x_0}(\Phi_2(x,\cdot))\circ d(x_0,y)\\
&=M_{x}(\Phi_2(\cdot,x_0))\circ d(x_0,y)\\
&=M_{x}(e^{-\rho b_v(x_0)}\frac{\partial}{\partial t} \varphi_2(\cdot,b_v(x_0))\circ d(x_0,y)\\
&=M_x(g)\circ d(x_0,y).
\end{align*}
Now by Theorem 3.8 in  \cite{PS15}  and integrating with respect to time we have for $t\in \R$
$$\varphi_2(x,t)=\int_0^{\lvert t \rvert } a^{-1}(M_x(g)\circ d(x_0,\cdot))(s)\,ds.$$
Since the shifted wave equation is linear we obtain a solution to the shifted wave equation with $\varphi(x,0)=f(x)$ and $\frac{\partial}{\partial t}\varphi(x,t)=g(x)$ by $\varphi=\varphi_1+\varphi_2$. This yields the claim. 
\end{proof}
\begin{folg}
From the characterisation in the Theorem \ref{thm:wave1} it follows now that $\varphi$ is a unique solution to the initial data $f,g$ as above. 
\end{folg}
Next we are going to show that a solution of the shifted wave equation has finite speed of propagation.
\begin{folg}\label{folg:support2}
Under the assumption of the Theorem \ref{thm:wave1} assume that $f,g$ have support in a geodesic ball of radius $R$ around $x_0\in X$ then
$$ \operatorname{supp} \varphi\subset \{(x,t)\in X\times\R\mid d(x_0,x)\leq R+\lvert t\rvert \}.$$
\end{folg}
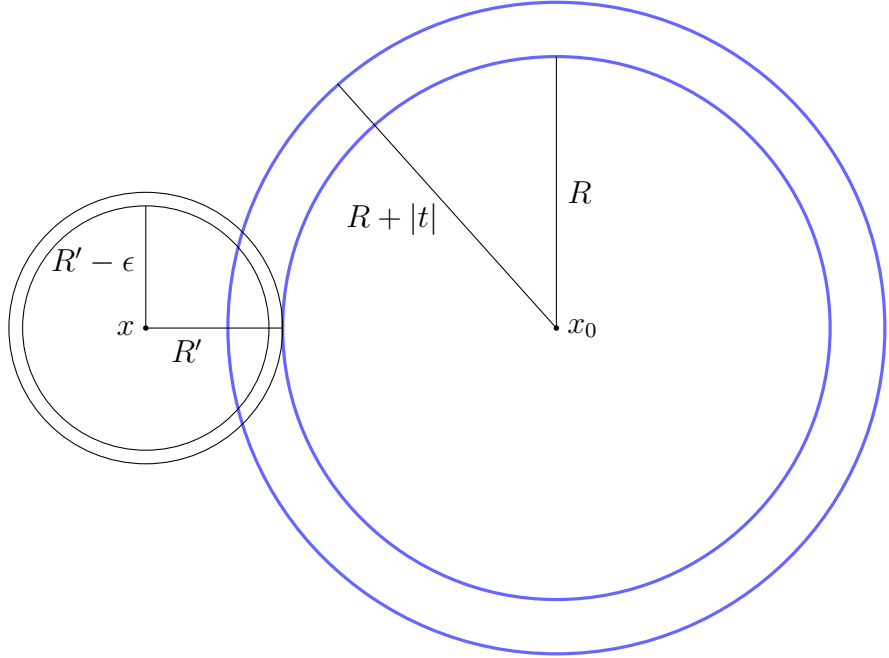
\begin{figure}[ht]
\centering
\begin{tikzpicture}[scale=1.8]
  
\draw[color=blue!60, very thick](0,0) circle (2cm);
\draw[color=blue!60, very thick](0,0) circle (2.4cm);
 \draw (-3 cm,0.0cm)circle (1cm);
  \draw (-3 cm,0.0cm)circle (0.9cm);
 \node[left] at (-3 cm,0.0cm){$x$};
  \fill[very thick,black]  (-3 cm,0.0cm)circle (0.02);

 \node[right] at (0,0) {$x_0$};

 \fill[very thick,black]  (0,0) circle (0.02);
\draw  (0 cm,0 cm)-- (0cm,2cm);
 \node[right] at (0 cm,1cm){$R$};
\draw  (0 cm,0 cm)-- (-1.6 cm,1.8cm);
 \node[below] at (-1.2 cm,1 cm){$R+\lvert t\rvert$};
\draw  (-3 cm,0.0 cm)-- (-2 cm,0.0 cm);
 \node[below] at (-2.7 cm,0 cm){$R'$};
 \draw  (-3 cm,0.0 cm)-- (-3cm,0.9 cm);
 \node[left] at (-3 cm,0.5 cm){$R'-\epsilon$};
 \end{tikzpicture}
   \caption{A sketch for the proof of Corollary \ref{folg:support2}.}
 \end{figure}

\begin{proof}
By Theorem \ref{thm:wave1} it is sufficient to prove that for $h\in C^{\infty}_c(X)$ with support $B(x_0,R)$ and $ d(x_0,x)> R+\lvert t\rvert $
\begin{align}\label{eq:defv}
v_x(t):=a^{-1}\big(M_x(h)\circ d(x_0,\cdot)\big)=0.
\end{align}
By the local injectivity of the dual Abel transform \cite[proof of Theorem 3.8]{PS15} we have that for $u:\R\to\R$ smooth and even
\begin{align}\label{eq:localinj}
a(u)\vert_{B(x_0,R)}=0 \Rightarrow u\vert_{[-R,R]}=0.
\end{align}
Now let $\epsilon>0$ arbitrary, $ d(x_0,x)> R+\lvert t\rvert $ and $R'=d(x_0,x)-R$ then
\begin{align}\label{eq:lemmasupp2}
a(v_x)(y)\overset{\text{(\ref{eq:defv})}}{=}M_x(h)\circ d(x_0,y)=0\quad \forall y\in B(x_0,R'-\epsilon).
\end{align}
Furthermore we have $R'=d(x_0,x)-R>\lvert t\rvert$ hence since $\epsilon>0$ is arbitrary we obtain from (\ref{eq:localinj}) and (\ref{eq:lemmasupp2}):
$$v_x(t)=0.$$
for all $(x,t)\in X\times \R$ with $d(x_0,x)> R+\lvert t\rvert $.
\end{proof}

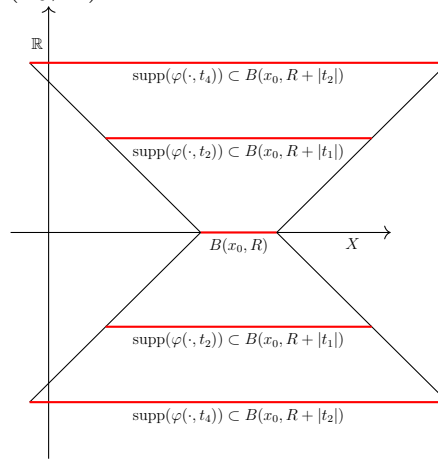
\begin{figure}[ht]
\caption{Finite propagation speed of a solution of the shifted wave equation with initial conditions supported in $B(x_0,R)$.}
\begin{tikzpicture}[scale=0.5, transform shape]
  \draw[->] (-3,0) -- (7,0);
  \draw[->] (-2,-6) -- (-2,6);
  \draw (-2.5,4.5)-- (2,0);
  \draw (8.5,4.5)-- (4,0);
  \draw[red,thick] (2,0)--(4,0);
      \draw[red,thick] (-0.5,2.5)--(6.5,2.5);
                  \draw[red,thick] (-2.5,4.5)--(8.5,4.5);
                  
                   \draw (-2.5,-4.5)-- (2,0);
  \draw (8.5,-4.5)-- (4,0);
      \draw[red,thick] (-0.5,-2.5)--(6.5,-2.5);
                  \draw[red,thick] (-2.5,-4.5)--(8.5,-4.5);

      \node [below] at (3,2.5) {$\operatorname{supp}(\varphi(\cdot,t_2))\subset B(x_0,R+\lvert t_1\rvert )$};
  
          \node [below] at (3,4.5) {$\operatorname{supp}(\varphi(\cdot,t_4))\subset B(x_0,R+\lvert t_2\rvert )$};
            \node [below] at (3,0) {$ B(x_0,R)$};

      \node [below] at (3,-2.5) {$\operatorname{supp}(\varphi(\cdot,t_2))\subset B(x_0,R+\lvert t_1\rvert )$};
       
          \node [below] at (3,-4.5) {$\operatorname{supp}(\varphi(\cdot,t_4))\subset B(x_0,R+\lvert t_2\rvert)$};
  \node [below] at (6,0) {$X$};
   \node [left] at (-2,5) {$\R$};

  \end{tikzpicture}
\end{figure}
\begin{bem}\label{folg:support}
The finite speed of propagation also follows from the general theory in \cite[Chapter 5]{MR0460898} or \cite[Chapter 2, Proposition 8.1]{TaylorPDE}  by choosing the canonical space time structure on $\R\times X$. See also \cite[Lemma 1.1]{BRANSON1991403}.
\end{bem}
 Next we provide an intrinsic prove of the existence of solution to the shifted wave equation without using general existence results mentioned in Remark \ref{rmk:xistenz}.
\begin{satz}\label{thm:existenzwave}
Let $f,g\in C^{\infty}_c(X)$ then the functions:
\begin{align*}
 \varphi_1(x,t)&=a^{-1}\big(M_x(f)\circ d(x_0,\cdot)\big)(\lvert t\vert )\\
 &\text{and}\\
 \varphi_2(x,t)&=\int_0^{\lvert t \rvert } a^{-1}(M_x(g)\circ d(x_0,\cdot))(s)\,ds
\end{align*}
are solutions of the shifted wave equation with initial condition 
\begin{align*}
\varphi_1(x,0)&=f(x)\\
\left.\frac{\partial}{\partial t}\right\vert_{t=0} \varphi_1(x,t)&=0\\
&\text{and}\\
\varphi_2(x,0)&=0\\
\left.\frac{\partial}{\partial t}\right\vert_{t=0} \varphi_2(x,t)&=g(x)
\end{align*}
respectively. 
Consequently $\varphi=\varphi_1+\varphi_2$ is a solution of the shifted wave equation with initial conditions 
 $\varphi(x,0)=f(x)$ and $\left.\frac{\partial}{\partial t}\right\vert_{t=0} \varphi(x,t)=g(x)$.
\end{satz}
\begin{proof}
Because $f$ and $g$ have compact support there exists an $R>0$ such that the support of $f$ and of $g$ is contained in the closed ball  $B(x_0,R)$.
We choose an orthonormal basis of eigenfunctions of the Dirichlet Laplacian on $B(x_0,2R)$, with respect to the $L^2$ norm on $B(x_0,2R)$, $\{\phi_k\}_{k\in\N}$ with $\Delta \phi_k=-\mu_k\phi_k$, $ 0\leq \mu_1\leq \mu_2\leq\cdots<\infty$ and $\mu_k=(\lambda_k^2+\rho^2)$ for some $\lambda_k\in \pm i[0,\rho]\cup \R$. First we observe that by Lemma \ref{equivalenz} for $x\in B(x_0,R)$ 
\begin{align}\label{eq:meanvalueseries}
M_x\phi_k(r)=\phi_k(x)\varphi_{\lambda_k}(r)\quad \forall r\leq R
\end{align}
 where $\varphi_{\lambda_k}$ is a eigenfunction of the operator $L_A$ (see Lemma \ref{equivalenz} for the definition) with $L_A \varphi_{\lambda_k}=-(\lambda_k^2+\rho^2)\varphi_{\lambda_k}$, $\varphi_{\lambda_k}(0)=1$ and $\lambda_k\in \pm i[0,\rho]\cup \R.$ 
Now we can represent $f$ and $g$ by a series in $\phi_k$: 
\begin{align*}
f(y)=\sum_{k=0}^{\infty} a_k\phi_k(y)\text{ and } g(y)=\sum_{k=0}^{\infty} b_k\phi_k(y),\forall y\in B(x_0,2R),\, a_k, b_k\in \C. 
\end{align*}
Using (\ref{eq:meanvalueseries}) we obtain for all $r\leq R$ and $x\in B(x_0,R)$
\begin{align*}
M_xf(r)=\sum_{k=0}^{\infty}a_k\phi_k(x)\varphi_{\lambda_k}(r)\text{ and  }  M_xg(r)=\sum_{k=0}^{\infty} b_k\phi_k (x)\varphi_{\lambda_k}(r).
\end{align*}
Applying the inverse dual Abel transform $a^{-1}$ yields, using that 
\begin{align*}
a^{-1}(\varphi_{\lambda_k}\circ d(x_0,\cdot))(\lvert t\rvert)&=a^{-1}(\varphi_{\lambda_k,x_0})(\lvert t\rvert )\\
&=\cos(\lambda_k t)
\end{align*}
 (see \cite[Proposition 3.4]{PS15}) and that $a^{-1}$ is linear, that:
\begin{align}
\label{eq:series1}a^{-1}\big(M_x(f)\circ d(x_0,\cdot)\big)(t)=\sum_{k=0}^{\infty}a_k\phi_k(x)\cos(\lambda_k t)\\
 a^{-1}\big(M_x(g)\circ d(x_0,\cdot)\big)(s)=\sum_{k=0}^{\infty}b_k\phi_k(x)\cos(\lambda_k s)\label{eq:series2}.
\end{align}
Therefore if we can show that (\ref{eq:series1}) converges uniformly in $x$ and $t$ we get:
\begin{align*}
\Delta \sum_{k=0}^{\infty}a_k\phi_k(x)\cos(\lambda_k t)&=\sum_{k=0}^{\infty}a_k\Delta\phi_k(x)\cos(\lambda_k t)\\
&=-\sum_{k=0}^{\infty}(\lambda_k^2+\rho^2)a_k\phi_k(x)\cos(\lambda_k t)
\end{align*}
and 
\begin{align*}
\frac{\partial^2}{\partial t^2}\sum_{k=0}^{\infty}a_k\phi_k(x)\cos(\lambda_k t)=-\sum_{k=0}^{\infty}\lambda_k^2a_k\phi_k(x)\cos(\lambda_k t).
\end{align*}
Hence $\varphi_1$ solves the shifted wave equation and satisfies the initial conditions $\varphi_1(x,0)=f$ and $\left.\frac{\partial}{\partial t}\right\vert_{t=0} \varphi_1(x,t)=0$ as one sees by (\ref{eq:series1}). 
Now suppose that (\ref{eq:series2}) converges uniformly in $x$ and $s$ then by integration we obtain:
\begin{align*}
\varphi_2(x,t)=\sum_{k=0}^{\infty}b_k\phi_k(x)\sin(\lambda_k t)\cdot\frac{1}{\lambda_k}
\end{align*}
where we interpret $\sin(\lambda_j t)\cdot\frac{1}{\lambda_j}=t$ if $\lambda_j=0$. 
Now applying the Laplacian yields:
\begin{align*}
\Delta \varphi_2(x,t)=-\sum_{k=0}^{\infty}(\lambda^2_k+\rho^2)b_k\phi_k(x)\cdot\sin(\lambda_kt)\frac{1}{\lambda_k}
\end{align*}
and we also get:
\begin{align*}
\frac{\partial^2}{\partial t^2} \varphi_2(x,t)=-\sum_{k=0}^{\infty}\lambda^2_kb_k\phi_k(x)\cdot\sin(\lambda_kt)\frac{1}{\lambda_k}.
\end{align*}
Therefore $\varphi_2$ satisfies the shifted wave equation, with the required initial conditions, as one can see by (\ref{eq:series2}).
Hence the proof would be complete if we show that (\ref{eq:series1}) and (\ref{eq:series2}) converge uniformly in both variables. This will follow from Lemma \ref{lemma:convergence}.
Under theses assumptions we have shown that $\varphi_1$ and $\varphi_2$ satisfy the theorem locally on the ball $B(x_0,R)$. If we now take $R'>R$ and repeat the construction above, we have by the local infectivity of the dual Abel transform  \cite[proof of Theorem 3.8]{PS15} that the series above coincide on $B(x_0,R)$. Therefore using the finite speed of propagation of the solution we can repeat the argument for a series $R_n\to \infty$  to obtain the theorem. 
\end{proof}
The lemma that finishes the proof of the theorem above is already contained in the proof of Theorem 3.8 in  \cite{PS15}.
\begin{lemma}\label{lemma:convergence}
Let $x_0\in X$, $R>0$  and $f\in C^{\infty}_c(X)$ such that the support of $f$ is contained in the closed ball $B(x_0,R)$ and  $\{\phi_k\}_{k\in\N}$ an orthonormal basis of eigenfunctions of the Dirichlet Laplacian on $B(x_0,R)$, with respect to the $L^2$ norm on $B(x_0,r)$ with $\Delta \phi_k=-\mu_k\phi_k$, $ 0\leq \mu_1\leq \mu_2\leq\cdots<\infty$ and $\mu_k=(\lambda_k^2+\rho^2)$ for some $\lambda_k\in \pm i[0,\rho]\cup \R.$ 
Furthermore let for $a_k\in \C$ the Fourier decomposition of $f$ be given by $f=\sum_{k=0}^{\infty}a_k\phi_k $ then the series 
$$ \sum_{k=0}^{\infty}a_k\phi_k(x)\lvert \lambda_k\rvert^m$$
converges uniformly in $x\in B(x_0,R)$. And hence all series in the proof of the Theorem \ref{thm:existenzwave} converge uniformly. 
\end{lemma}
\begin{proof}
First we observe that by the Sobolev embedding theorem (see for instance \cite[Chapter 3]{hebey1996sobolev}) 
there exists a constant $C_0>0$, such that for every function $u$ in the Sobolev space $H^2_{2n}(B(x_0,R))$ we have:
\begin{align}\label{eq:sobolev}
\lVert u\rVert_{\text{sup}}\leq C_o\big(\lVert u\rVert_{L^2(B(x_0,R))} +\lVert \Delta^n u\rVert_{L^2(B(x_0,R))}\big),
\end{align}
 where $\lVert \cdot\rVert_{\text{sup}}$ is the sup norm on $C^0(B(x_0,R))$ and $n=\operatorname{dim}X$. 
Now since $\phi_k$ is an orthonormal basis with respect to the $L^2$ norm on $B(x_0,R)$ we have 
\begin{align*}
\lvert \phi_k(x)\rvert \leq \lVert \phi_k\rVert_{\text{sup}}\overset{\text{(\ref{eq:sobolev})}}{\leq} C_0(1+\mu_k^n),\quad \forall x\in B(x_0,R).
\end{align*}
By Weyl's law (see for instance \cite[p.155]{chavel1984eigenvalues}) we obtain that $k\sim \mu_k^{n/2}$, meaning that for $k>0$ there is a constant $C\geq 1$ such that $\frac{1}{C}\leq \frac{\mu_k^{n/2}}{k}\leq C$. Therefore there is a $k_0\in \N$ such that for some $C_1>0$ 
\begin{align*}
C_1(1+\mu_k^n)\leq C_1k^2 \quad \forall k>k_0.
\end{align*}
This yields: 
\begin{align}\label{eq:unifom1}
\lvert \phi_k(x)\rvert \leq \lVert \phi_k\rVert_{\text{sup}}\leq C_1k^2 \quad \forall k>k_0.
\end{align}
Now observe that $f\in C^{\infty}_c(X)$ with support contained in $B(x_0,R)$ hence $\Delta^j f\in C^{\infty}_c(X)$ for every $j\in \N$  and has support in $B(x_0,R)$. Therefore:
\begin{align*}
\Delta^jf=\sum_{k=0}^{\infty}a_k\mu_k^j\phi_k
\end{align*}
converges uniformly on $B(x_0,R)$
and $\Delta^jf \in L^2(B(x_0,R))$.
This yields since $\{\phi_k\}_{k\in \N}$ is a orthonormal basis with respect to the $L^2$ norm
\begin{align*}
\infty> \lVert \Delta^jf\rVert^2_2=\sum_{k=0}^{\infty}\lvert a_k\rvert^2 \mu_k^{2j}.
\end{align*}
Now $\mu_k=(\lambda_k^2+\rho^2)$ hence:
\begin{align}\label{eq:unifom2}
\infty>\sum_{k=0}^{\infty}\lvert a_k\rvert^2(\lambda_k^2+\rho^2)^{2j}\geq \sum_{k=0}^{\infty}\lvert a_k\rvert^2(\lambda_k)^{4j} \quad\forall j\in \N.
\end{align}
With this we obtain for $l\in \N$ arbitrarily and any $x\in B(x_0,R)$:
\begin{align*}
\sum_{k=0}^{\infty}\lvert a_k\rvert \lvert \phi_k(x)\lvert \lambda_k\rvert^m\overset{\text{(\ref{eq:unifom1})}}&{\leq}C_1 \sum_{k=0}^{\infty}\lvert a_k\rvert k^2\lvert \lambda_k\rvert ^m\\
=&C_1\sum_{k=0}^{\infty}\lvert a_k\rvert k^2 \lvert \lambda_k\rvert^{m+l}\lvert \lambda_k\rvert^{-l}\\
\overset{\text{Cauchy Schwarz}}&{\leq}C_1\Big(\sum_{k=0}^{\infty}\lvert a_k\rvert^2k^2\lvert \lambda_k\rvert^{2m+2l}\Big)^{1/2}\\&\quad\quad\quad\cdot \Big(\sum_{k=0}^{\infty}\lvert \lambda_k\rvert^{-2l}\Big)^{1/2}.
\end{align*}
Now using Weyl's law and $\mu_k=\lambda^2_k+\rho^2$ we conclude: 
\begin{align*}
C_1&\Big(\sum_{k=0}^{\infty}\lvert a_k\rvert^2k^2\lvert \lambda\rvert^{2m+2l}\Big)^{1/2}\cdot \Big(\sum_{k=0}^{\infty}\lvert \lambda_k\rvert^{-2l}\Big)^{1/2}\\
&\leq C_1\Big(\sum_{k=0}^{\infty}\lvert a_k\rvert^2 \lvert\lambda_k\rvert^{2(m+l+2n)}\Big)^{1/2} \cdot \Big(\sum_{k=0}^{\infty}\lvert \lambda_k\rvert^{-2l}\Big)^{1/2}.
\end{align*}
Now with $l=n$ we have
\begin{align*}
\sum_{k=0}^{\infty}\lvert a_k\rvert^2\lvert \lambda_k\rvert^{2(m+4n)}\overset{\text{(\ref{eq:unifom2})}}{<}\infty
\end{align*}
and using Weyl's law there is a constant $C_2$ such that:
\begin{align*}
\sum_{k=0}^{\infty}\lvert \lambda_k\rvert^{-2n}\leq C_2\cdot\sum_{k=0}^{\infty}\frac{1}{k^2}<\infty.
\end{align*}
This yields the claim. 
\end{proof}

\begin{bem}\label{rmk:xistenz}
It also follows from the abstract theory of PDE,s that the solution of the shifted wave equation exist. See for instance \cite[Chapter 2+6]{TaylorPDE},  \cite[Chapter 5+6]{MR0460898}, \cite[Chapter 3]{MR2298021} and \cite{MR946226}. In their context one would consider the product manifold $\R\times X$ with the canonical space time structure where the shifted wave equation corresponds to a lower order perturbation  of the ordinary wave equation.
\end{bem}

\section{The rank one case}
A non compact simply connected harmonic manifold $X$ is said to be of purely exponential volume growth if there exists some constants $C\geq1$ and $\rho>0$ such that:
\begin{align*}
\frac{1}{C}\leq \frac{A(r)}{e^{2\rho r}}\leq C.
\end{align*}
This property is by \cite{knieper2009new} equivalent to
\begin{itemize}
\item The Geodesic Flow in $SX$ is  Anosov with respect to the Sasaki metric
\item Gromov Hyperbolicity 
\item Rank one.
\end{itemize}
Note that non positive curvature implies purely exponential volume growth.

From now on let $(X,g)$ to be a non compact simply connected harmonic manifold of rank one.
The geometric boundary $\partial X$ is defined by equivalence classes of geodesic rays. Where two rays are equivalent if their distance is bounded. 
The topology on $\partial X$ is the cone topology with the property that for $\overline{X}=X\cup\partial X$ and $B_1(x)=\{v\in T_xX\vert~\Vert v\Vert\leq 1\}$ the map $pr_x: B_1(x)\to \overline{X}$
\begin{align*}
pr_x (v)=\begin{cases}
\gamma_v(\infty)&\text{if}~\Vert v\Vert=1\\
\exp(\frac{1}{1-\Vert v \Vert}v)&\text{if}~\Vert v\Vert<1
\end{cases}
\end{align*}
is a homeomorphism.
It turns out that since the geodesic flow is Anosov the Busemann function only depends on the direction of the ray. Hence 
for $x\in X$ and $\xi\in\partial X$ being the point at infinity of the  geodesic $\gamma$ we can alternatively define the  Busemann function $B_{\xi,x}:X\to\R$ by $$B_{\xi,x}(y)=\lim_{t\to\infty}(d(y,\gamma(t))-d(x,\gamma(t)).$$

Furthermore we obtain a cocycle property:
\begin{align}\label{coBuse}
B_{\xi,x}=B_{\xi,\sigma}-B_{\xi,\sigma}(x).
\end{align}
By the above if $v\in S_{\sigma}X$ defines the unique geodesic ray such that $c_v(\infty)=\xi$ then
 $$b_v(x)=B_{\xi,\sigma}(x)\quad \forall x\in X.$$
 For a proof see \cite[Lemma 2.2]{Biswas2019}.
% \end{bem}
%\end{lemma}
With this we have $\Delta B_{\xi,\sigma} =2\rho$ where $2\rho$ is the mean curvature of the horospheres. And obtain:
 $g(y)=e^{(i\lambda-\rho)B_{\xi, x}(y)}$ is a eigenfunction of the Laplacian with $g(x)=1$ and $\Delta g= -(\lambda^2+\rho^2)g$ for $\lambda\in \C$.
Furthermore, by pushing forward  the probability  measure induced by the metric $\theta_x$ on $S_x X$ under $pr_x$  we obtain a  probability measure $\mu_x$ on $\partial X$. Hence, we have a family of probability measures $\{\mu_x\}_{x\in X}$, that are pairwise absolutely continuous with Radon-Nikodym derivative  
\begin{align}\label{eq:RNdiv}
\frac{d\mu_x}{d\mu_y}(\xi)=e^{-2\rho B_{\xi,x}(y)}.
\end{align}
 For a detailed proof see \cite[Theorem 1.4]{Knieper2016}.

\subsection{Fourier Transform and Plancherel Theorem on Rank One Harmonic Manifolds}
The main tool in defining the Fourier transform on rank one harmonic manifolds is the theory of hypergroups. This was first presented for harmonic manifolds with pinched negative curvature in \cite{biswas2018fourier} and then extended in \cite{Biswas2019} to rank one harmonic manifold.
Since we refrain ourselves from details, we refer the reader 
 to \cite{bloom2011harmonic} for a thorough discussion of the topic and the definition. In \cite[Section 4.2]{Biswas2019} the authors showed that the density function $A(r)$ of a  harmonic manifold of rank one  satisfies the following conditions   
\begin{itemize}
\item[(C1)] $A$ is increasing and $A(r)\to\infty$ for $r\to \infty$.
\item[(C2)] $\frac{A^\prime}{A}$ is decreasing and $\rho=\frac{1}{2} \lim\limits_{r\to\infty} \frac{A^\prime(r)}{A(r)}>0$.
\item[(C3)] For $r>0$, $A(r)=r^{2\alpha +1 }B(r)$ for some $\alpha>-\frac{1}{2}$ and some even $C^{\infty}$ function $B(x)$ on $\R$ with $B(0)=1$.
\item[(C4)] 
\begin{align*}
G(r)=\frac{1}{4}\Bigl(\frac{A^\prime}{A}(r)\Bigr)^2 +\frac{1}{2}\Bigl(\frac{A^\prime}{A}(r)\Bigr)^\prime -\rho^2
\end{align*}\label{C42}
is bounded on $[r_0,\infty)$ for all $r_0>0$  and 
    \begin{align*}
\int_{r_1}^\infty r\vert G(r)\vert \,dr<\infty\quad\text{for some}~r_1 >0.
\end{align*}
\end{itemize}
And therefore $A(r)$ defines a   Ch\'ebli-Trim\'eche hypergoup. The structure is  of the so defined hypergroup is related to the second order differential operator given by the radial part of the Laplacian: 
\begin{align}\label{eq:A}
 L_{A}=\frac{d^2}{dr^2}+\frac{A^{\prime}(r)}{A(r)}\frac{d}{dr}.
\end{align}
Let
\begin{align}
\varphi_{\lambda}:\R^+\to\R, \quad\lambda\in [0,\infty)\cup [0,i\rho]
\end{align}
be the eigenfunction of $L_A$ with
\begin{align}\label{eq:eigenlaplace}
L_{A}\varphi_{\lambda}=-(\lambda^2+\rho^2)\varphi_\lambda
\end{align}
and which admits a smooth extension to zero with $\varphi_{\lambda}(0)=1$.
 Under  conditions (C1)-(C4) it was shown in \cite{Bloom1995TheHM}  that there is a complex function $\mathbf{c}$ on $\C\setminus \{0\}$. Such that 
for the two linear independent solutions of $$L_{A}u=-(\lambda^2+\rho^2)u$$ 
$\Phi_{\lambda}$ and $\Phi_{-\lambda}$ which are asymptotic to exponential functions i.e.
\begin{align}\label{eq:symeigen}
\Phi_{\pm \lambda}(r)= e^{(\pm i\lambda-\rho)r}(1+o(1))\text{ as }r\to\infty
\end{align}
we have
\begin{align}\label{eigende}
\varphi_{\lambda}=\mathbf{c}(\lambda)\Phi_{\lambda}+\mathbf{c}(-\lambda)\Phi_{-\lambda}\quad\forall \lambda\in \C\setminus \{0\}.
\end{align}
Imposing the additional condition that $\lvert \alpha \rvert>\frac{1}{2}$  the authors in \cite{Bloom1995TheHM}  showed that $\mathbf{c}$-function dose not have zeros on the closed lower half plane. Hence this would exclude the case $\operatorname{dim} X=3$ (see \cite{Biswas2019}) but the Lichnerowicz conjecture is affirmed in the case $\operatorname{dim} X<6$ and therefore the Jacobin analysis applies, and we can use the $\mathbf{c}$-function obtained in this context.
We then can define the radial Fourier transform by: 
\begin{defi}
Let $f:X\to\C$ be, i.e. $f=u\circ d_{\sigma}$ for some $\sigma \in X$, where $u:[0,\infty)\to\C$ and $d_{\sigma}:X\to \R$ is the distance function. The radial Fourier transform of $f$ is given by: 
$$\widehat{f}(\lambda):=\widehat{u}(\lambda)=\int_0^{\infty}u(r)\varphi_{\lambda}(r)A(r)\,dr.$$
\end{defi}
Note that in the following we will omit to mention the base point $\sigma$ unless there is the possibility of confusion. For $f$ radial around $\sigma\in X$, % we mean that  $f$ is radial around some point 
we will use  $\sigma$ as  base  point for the radial Fourier transform unless stated otherwise. 
Now observe that we obtain the radial eigenfunctions of the Laplace operator with eigenvalue $-(\lambda^2+\rho^2)$ by:
\begin{align}\label{defi:ef}
\varphi_{\lambda,\sigma}(y)=\varphi_{\lambda}\circ d(\sigma,y)\quad \forall x,y\in X.
\end{align}
Using the results from \cite{Bloom1995TheHM}  the authors in \cite{Biswas2019} showed that there is a constant $C_0$ such that 
for $f\in L^1(X)$ radial, i.e. $f=u\circ d_{\sigma}$ for some $\sigma\in X$ and $u:[0,\infty)\to\R$ such that $\widehat{u}\in L^1((0,\infty),C_0\lvert \mathbf{c}(\lambda)\rvert^{-2}\,d\lambda)$.
\begin{align}\label{fourad}
f(y)=C_0\int_{0}^\infty \widehat{f}(\lambda)\varphi_{\lambda,\sigma}(y)\vert \mathbf{c}(\lambda) \vert^{-2}\, d\lambda.
\end{align}
Moreover the radial Fourier transform extends to an isometry between the $L^2$-radial functions denoted by $L^2(X,\sigma)$ and $$L^2((0,\infty),C_0\lvert \mathbf{c}(\lambda)\rvert^{-2}\,d\lambda).$$
See  \cite[Theorem 4.7]{Biswas2019}.
%%%%%%%%%%%%%%%%%%%%%%%%%%%%%%%%%%%%%%%%%%%%%%%%%%%%%%%%%%%%%%%%%%%%%%
 In the same fashion as in the case of the Helgason Fourier transform on symmetric spaces we can extend the Fourier transform to non radial functions. By using radial symmetry of the Poisson kernel. Again the main reference for this is \cite{Biswas2019}.
\begin{defi}\label{def:fourier}
		Let $\sigma\in X$ for $f:X\to \C$ measurable, the Fourier transform of $f$ based at $\sigma$ is given by 		
		$$\tilde{f}^{\sigma}(\lambda,\xi)=\int_{X}f(y)e^{(-i\lambda-\rho)B_{\xi,\sigma}(y)}\,dy$$
		for $\lambda\in\C$, $\xi\in\partial X$ for which the integral above converges.		
\end{defi}
		We can immediately note that because of the cocycle property of the Busemann function (\ref{coBuse})
	
		we obtain:
		\begin{lemma}
		Let $f\in C^{\infty}_c(X)$ and $x,\sigma\in X$ then we have:
		\begin{align}\label{eq:fourpoint}
			\tilde{f}^x (\lambda,\xi)=e^{(i\lambda+\rho)B_{\xi,\sigma}(x)}\tilde{f}^{\sigma}(\lambda,\xi).
			\end{align}
		\end{lemma}
		\begin{proof}	
		Let $x,\sigma\in X$ and $f\in C^{\infty}_c(X)$ then we have for $\lambda\in \C$ and $\xi\in \partial X$ that:
		\begin{align*}
		\tilde{f}^x (\lambda,\xi)&=\int_Xf(y)e^{(-i\lambda-\rho)B_{\xi,x}(y)}\,dy\\
		&\overset{\text{(\ref{coBuse})}}{=}\int_Xf(y)e^{(-i\lambda-\rho)B_{\xi,\sigma}(y)}\cdot e^{(i\lambda+\rho)B_{\xi,\sigma}(x)}\,dy\\
		&=e^{(i\lambda+\rho)B_{\xi,\sigma}(x)}\int_Xf(y)e^{(-i\lambda-\rho)B_{\xi,\sigma}(y)}\,dy\\
		&=e^{(i\lambda+\rho)B_{\xi,\sigma}(x)}\tilde{f}^{\sigma}(\lambda,\xi).
		\end{align*}
		\end{proof}
		Furthermore the Fourier transform coincides with the radial Fourier transform on radial functions. For details see \cite[Lemma 5.2]{Biswas2019}. 	
	The inversion formula follows now from the representation of the radial eigenfunctions via convex combination of non radial eigenfunctions, \cite[Theorem 5.6]{Biswas2019},:
			\begin{align}\label{radialeigen}
			\varphi_{\lambda,\sigma}(y)=\int_{\partial X}e^{(i\lambda-\rho)B_{\xi,\sigma}(y)}\,d\mu_{\sigma} (\xi)\quad \forall\sigma\in X.
		\end{align}
This is analogous to the well known formula on a rank one symmetric space $G/K$ and harmonic $NA$ groups. 
See for the symmetric case \cite[Chapter III, Section 11]{helgason1994geometric} and for the harmonic $NA$ group \cite{Damek1992} and \cite{fourierNA2}.
		Using equation (\ref{radialeigen}) the authors obtain:
			\begin{align}
			f(x)=C_0 \int_{0}^{\infty}\int_{\partial X}\tilde{f}^{\sigma}(\lambda,\xi)e^{(i\lambda-\rho)B_{\xi,\sigma}(x)}\,d\mu_{\sigma}(\xi)
			\vert \mathbf{c}(\lambda)\vert^{-2}\,d\lambda,
			\end{align}
		where $C_0$ is the same constant given in (\ref{fourad}).
		Additionally the authors obtain a Plancherel theorem:

\begin{satz}[\cite{Biswas2019}]\label{Placherel Theorem}
		Let $\sigma\in X$ and $f,g\in C_{c}^{\infty}(X)$. Then we have:
			$$\int_X f(x)\overline{g(x)}\,dx=C_{0}\int_{0}^{\infty}\int_{\partial X}\tilde{f}^{\sigma}(\lambda,\xi)\overline{\tilde{g}
			^{\sigma}(\lambda,\xi)} \vert \mathbf{c}(\lambda)\vert^{-2}\,d\mu_{\sigma}(\xi)d\lambda 
			$$ 
		and the Fourier transform extends to an isometry between $$L^2(X)$$ and
		 $$L^2((0,\infty)\times \partial X,C_0 \vert \mathbf{c}(\lambda)\vert^{-2} \,d\mu_{\sigma}(\xi)\,d\lambda).$$
\end{satz}

 \subsection{Wave Equation Under Fourier Transform and conservation of Energy}
 Using the Fourier transform we can obtain the conservation of energy for solutions of the wave equation similar to the result  in \cite{AMBP_2010__17_2_327_0} for Damek-Ricci spaces. 
 For this we first need to study the action of the Laplacian under Fourier transform. 
 \begin{lemma}\label{lemma:fourierdelta}
 Let $f\in L^2(X)$ such that $\Delta f \in L^2(X)$, where $\Delta f$ is meant in the sense of distributions i.e. $\Delta f$ is defined by 
 $$\int_X \Delta f(x)g(x)\,dx:=\int_Xf(x)\Delta g(x)\,dx\quad\forall g\in C^{\infty}_c(x),$$
  and $\sigma\in X$ then:
 \begin{align*}
 \widetilde{\Delta f}^{\sigma}(\lambda,\xi)=-(\lambda^2 +\rho^2) \widetilde{ f}^{\sigma}(\lambda,\xi)
 \end{align*} 
for almost every $(\lambda,\xi)\in (0,\infty)\times\partial X$.
 \end{lemma}

 \begin{proof}
 Let $\sigma\in X$. 
 Since $C^{\infty}_c(X)$ is dense in $L^2(X)$ and by using the Plancherel theorem it is sufficient to prove the assertion for $f\in C^{\infty}_c(X)$. To be more precise: If $f,\Delta f\in L^2(X)$ then there is a sequence $f_n\in C^{\infty}_c(X)$ such that $f_n\to f$ and $\Delta f_n\to \Delta f$ in $L^2(X)$. For this see \cite[Corollary 2.5]{STRICHARTZ198348}. 
Let $\sigma\in X$ then the above implies by the the Plancherel theorem that $\tilde{f_n}^{\sigma}\to\tilde{f}^{\sigma}$  and $\widetilde{\Delta f_n}^{\sigma}\to \widetilde{\Delta f}^{\sigma}$ in $L^2((0,\infty)\times \partial X,C_0 \vert \mathbf{c}(\lambda)\vert^{-2} \,d\mu_{\sigma}(\xi)\,d\lambda)$. Therefore we find subsequences such that both converge point wise almost everywhere. 

  Then since the Laplacian is essentially self adjoint and 
 \begin{align*}
 \Delta e^{(-i\lambda-\rho)B_{\xi, \sigma}(y)}=-(\lambda^2 +\rho^2) e^{(-i\lambda-\rho)B_{\xi, \sigma}(y)}\quad \forall y\in X
 \end{align*}
  we have almost every where:
 \begin{align*}
  \widetilde{\Delta f_n}^{\sigma}(\lambda,\xi)&=\int_X \Delta f_n(x)e^{(-i\lambda-\rho)B_{\xi, \sigma}(x)}\,dx\\
  &=\int_X  f_n(x)\Delta e^{(-i\lambda-\rho)B_{\xi, \sigma}(x)}\,dx\\
  &=-(\lambda^2+\rho^2) \int_X  f_n(x)e^{(-i\lambda-\rho)B_{\xi, \sigma}(x)}\,dx\\
  &=-(\lambda^2+\rho^2)  \tilde{f_n}^{\sigma}(\lambda,\xi).
 \end{align*}
 Therefore we have after if necessary passing to a subsequences that
\begin{align*} 
 -(\lambda^2 +\rho^2) \widetilde{ f}^{\sigma}(\lambda,\xi)&=\lim_{n\to\infty} -(\lambda^2 +\rho^2) \widetilde{ f_n}^{\sigma}(\lambda,\xi) \\
& =\lim_{n\to\infty}  \widetilde{\Delta f_n}^{\sigma}(\lambda,\xi)\\
& =\widetilde{\Delta f}^{\sigma}(\lambda,\xi)\\
 \end{align*}
 almost everywhere.
 \end{proof}

\begin{satz}\label{thm:wavefour}
Suppose $(X,g)$ is a harmonic manifold of rank one. Let $\sigma\in X$ then the Fourier transform of a $C^{\infty}$ solution to the shifted wave equation $\varphi:X\times \R\to\C$  with initial conditions
\begin{align*}
\varphi(x,0)&=f(x)\in C^{\infty}_c(X),\\
 \left.\frac{\partial}{\partial t}\right\vert_{t=0} \varphi(x,t)&=g(x)\in C^{\infty}_c(X)
\end{align*}
is given by
\begin{align*}
\varphi(x,t)=C_0 \int_{0}^{\infty}\int_{\partial X}\big(\tilde{f}^{\sigma}(\lambda,\xi)\cos(\lambda t)+\tilde{g}^{\sigma}(\lambda,\xi)\frac{\sin(\lambda t)}{\lambda}\big)\\
\cdot e^{(i\lambda-\rho)B_{\xi,\sigma}(x)}\,d\mu_{\sigma}(\xi)\lvert \mathbf{c}(\lambda)\rvert^{-2}\,d\lambda.
\end{align*}
\end{satz}
\begin{proof}
Since by Remark \ref{folg:support}    $\varphi(\cdot,t)$  and all its derivatives in $t$ have compact support for every $t\in \R$
we obtain:
\begin{align*}
\frac{\partial^2}{\partial t^2}\widetilde{\varphi}^{\sigma}((\lambda,\xi);t)&=\frac{\partial^2}{\partial t^2}\int_X \varphi(x)e^{(-i\lambda-\rho)B_{\xi,\sigma}(x)}\,dx\\
&=\int_X \frac{\partial^2}{\partial t^2} \varphi(x)e^{(-i\lambda-\rho)B_{\xi,\sigma}(x)}\,dx\\
&=\widetilde{\frac{\partial^2}{\partial t^2}\varphi}^{\sigma}((\lambda,\xi);t)\\
&=\widetilde{\Delta\varphi}^{\sigma}((\lambda,\xi);t)+\rho^2\widetilde{\varphi}^{\sigma}((\lambda,\xi);t)\\
\overset{\text{Lemma \ref{lemma:fourierdelta} }}&{=}-(\lambda^2-\rho^2)\widetilde{\varphi}^{\sigma}((\lambda,\xi);t)+\rho^2\widetilde{\varphi}^{\sigma}((\lambda,\xi);t)\\
&=-\lambda^2\widetilde{\varphi}^{\sigma}((\lambda,\xi);t).
\end{align*}
Now the wave equation becomes:
\begin{align*}
\frac{\partial^2}{\partial t^2}\widetilde{\varphi}^{\sigma}((\lambda,\xi);t)&=-\lambda^2 \widetilde{\varphi}^{\sigma}((\lambda,\xi);t)\\
\widetilde{\varphi}^{\sigma}((\lambda,\xi);0)&=\tilde{f}^{\sigma}(\lambda,\xi)\\
\frac{\partial}{\partial t}\widetilde{\varphi}^{\sigma}((\lambda,\xi);0)&=\tilde{g}^{\sigma}(\lambda,\xi)
\end{align*}
hence
$$\widetilde{\varphi}^{\sigma}((\lambda,\xi);t)=\tilde{f}^{\sigma}(\lambda,\xi)\cos(\lambda t)+\tilde{g}^{\sigma}(\lambda,\xi)\frac{\sin(\lambda t)}{\lambda},$$
therefore applying the inverse Fourier transform yields the claim. 
\end{proof}
\begin{bem}
While the representation of the solutions of the shifted wave equation from Theorem \ref{thm:wave1} corresponds to the classical representation of the solutions of the wave equation on $\R^n$ by \'{A}sgeirsson \cite{MR1513094} the representation obtained in Theorem \ref{thm:wavefour} corresponds to the operator expression for the operator $\Delta_{\rho}:=\Delta+\rho^2$:
$$\varphi(x,t)=\cos(\sqrt{-\Delta_{\rho}}t)f(x)+\frac{\sin(\sqrt{-\Delta_{\rho}}t)}{\sqrt{-\Delta_{\rho}}}g(x).$$
In turn this again corresponds to the expression of the solution as a power series in the proof Theorem \ref{thm:existenzwave}.
\end{bem}
\begin{defi}
Let  $\varphi:X\times \R\to\C$  be a solution of the shifted wave equation, we define its kinetic energy $\mathcal{K}(\varphi)$ by: 
$$\mathcal{K}(\varphi)(t):=\frac{1}{2}\int_X \left\lvert \frac{\partial}{\partial t } \varphi(x,t)\right \rvert^2 \,dx$$
and its potential energy $\mathcal{P}(\varphi)(t)$ by
$$ \mathcal{P}(\varphi)(t):=\frac{1}{2} \int_X \varphi(x,t)(-\Delta -\rho^2)\overline{\varphi}(x,t)\,dx.$$
The total energy is defined by  $$\mathcal{E}(\varphi)(t):= \mathcal{K}(\varphi)(t)+\mathcal{P}(\varphi)(t).$$
\end{defi}
\begin{lemma}\label{lemma:energy}
Suppose $(X,g)$ is a harmonic manifold of rank one. Let $\sigma\in X$ and $\varphi:X\times \R\to\C$  be a solution to the shifted wave equation  with initial conditions 
\begin{align*}
\varphi(x,0)&=f(x)\in C^{\infty}_c(X)\\
\left. \frac{\partial}{\partial t}\right\vert_{t=0} \varphi(x,t)&=g(x)\in C^{\infty}_c(X)
\end{align*}
 then we have 
\begin{align}
2\mathcal{K}(\varphi)(t)=&C_0\int_0^{\infty}\int_{\partial X}\lvert -\lambda \widetilde{f}^{\sigma}(\lambda,\xi)\sin(\lambda t)\label{eq:energi1}\\
&+\widetilde{g}^{\sigma}(\lambda,\xi)\cos(\lambda t)\rvert^2 \,d\mu_{\sigma}(\xi)\lvert \mathbf{c}(\lambda)\rvert^{-2}\,d\lambda\nonumber
\end{align}
and 
\begin{align}
2\mathcal{P}(\varphi)(t)=&C_0\int_0^{\infty}\int_{\partial X}\lvert \lambda \widetilde{f}^{\sigma}(\lambda,\xi)\cos(\lambda t)\label{eq:energi2}\\
&+\widetilde{g}^{\sigma}(\lambda,\xi)\sin(\lambda t)\rvert^2 \,d\mu_{\sigma}(\xi)\lvert \mathbf{c}(\lambda)\rvert^{-2}\,d\lambda.\nonumber
\end{align}
\end{lemma}
\begin{proof}
Using the Plancherel theorem for the Fourier transform and Theorem \ref{thm:wavefour}  we obtain for the kinetic energy
\begin{align*}
2\mathcal{K}(\varphi)(t)&=\int_X \left\lvert \frac{\partial}{\partial t } \varphi(x,t)\right \rvert^2 \,dx\\
\overset{\text{Plancherel theorem}}&{=}C_0\int_0^{\infty}\int_{\partial X}\left\lvert \frac{\partial}{\partial t} \widetilde{\varphi}^{\sigma}(\lambda,\xi;t)\right\rvert^2 \,d\mu_{\sigma}(\xi)\lvert \mathbf{c}(\lambda)\rvert^{-2}\,d\lambda\\
\overset{\text{Theorem \ref{thm:wavefour} }}&{=}C_0\int_0^{\infty}\int_{\partial X}\lvert -\lambda \widetilde{f}^{\sigma}(\lambda,\xi)\sin(\lambda t)\\
&+\tilde{g}^{\sigma}(\lambda,\xi)\cos(\lambda t) \rvert^2 \,d\mu_{\sigma}(\xi)\lvert \mathbf{c}(\lambda)\rvert^{-2}\,d\lambda.
\end{align*}
For the potential energy we are using the Plancherel theorem for the Fourier transform, Theorem \ref{thm:wavefour} and Lemma \ref{lemma:fourierdelta}:

\begin{align*}
2\mathcal{P}(\varphi)(t)=&\int_X \varphi(x,t)(-\Delta -\rho^2)\overline{\varphi}(x,t)\,dx\\
\overset{\text{Plancherel theorem}}&{=}C_0\int_0^{\infty}\int_{\partial X}  \widetilde{\varphi}^{\sigma}(\lambda,\xi;t)\\
&\cdot\big(-\overline{\widetilde{\Delta \varphi}^{\sigma}}(\lambda,\xi;t)-\overline{\widetilde{\rho^2 \varphi}^{\sigma}}(\lambda,\xi;t)\big)  \,d\mu_{\sigma}(\xi)\lvert  \mathbf{c}(\lambda)\rvert^{-2}\,d\lambda\\
\overset{\text{Lemma \ref{lemma:fourierdelta}}}&{=}C_0\int_0^{\infty}\int_{\partial X}  \widetilde{\varphi}^{\sigma}(\lambda,\xi;t)\\
&\cdot\big((\lambda^2+\rho^2)\overline{\widetilde{ \varphi}^{\sigma}}(\lambda,\xi;t)-\overline{\widetilde{\rho^2 \varphi}^{\sigma}}(\lambda,\xi;t)\big) \,d\mu_{\sigma}(\xi) \lvert   \mathbf{c}(\lambda)\rvert^{-2}\,d\lambda\\
\overset{\text{Theorem \ref{thm:wavefour}}}&{=}C_0\int_0^{\infty}\int_{\partial X}\lvert \lambda \widetilde{f}^{\sigma}(\lambda,\xi)\cos(\lambda t)\\
&+\widetilde{g}^{\sigma}(\lambda,\xi)\sin(\lambda t)\rvert^2 \,d\mu_{\sigma}(\xi)\lvert \mathbf{c}(\lambda)\rvert^{-2}\,d\lambda.
\end{align*}
\end{proof}
\begin{satz}\label{thm:energy}
Suppose $(X,g)$ is a harmonic manifold of rank one. Let $\sigma\in X$ and $\varphi:X\times \R\to\C$  a solution to the shifted wave equation  with initial conditions $f,g\in C^{\infty}_c(X)$ then the total energy $\mathcal{E}(\varphi)(t)$ is independent of $t$. In particular
\begin{align*}
2\mathcal{E}(\varphi)(t)=&\lVert \lambda \tilde{f}^{\sigma}\rVert^2_{L^2((0,\infty)\times \partial X,C_0 \vert \mathbf{c}(\lambda)\vert^{-2} \,d\mu_{\sigma}(\xi)\,d\lambda)}\\
&+\lVert \tilde{g}^{\sigma}\rVert^2_{L^2((0,\infty)\times \partial X,C_0 \vert \mathbf{c}(\lambda)\vert^{-2} \,d\mu_{\sigma}(\xi)\,d\lambda)}.
\end{align*}
\end{satz}
\begin{proof}
If we look at the terms under the integrals in Lemma \ref{lemma:energy} separately we obtain:
\begin{align*}
(\ref{eq:energi1})=&\lvert -\lambda \widetilde{f}^{\sigma}(\lambda,\xi)\sin(\lambda t)+\widetilde{g}^{\sigma}(\lambda,\xi)\cos(\lambda t)\rvert^2\\
=& \lambda^2 \lvert\widetilde{f}^{\sigma}(\lambda,\xi)\rvert^2\sin^2(\lambda t)+\lvert\widetilde{g}^{\sigma}(\lambda,\xi)\rvert^2\cos^2(\lambda t)\\
&-\lambda \widetilde{f}^{\sigma}(\lambda,\xi)\sin(\lambda t)\cdot\overline{\widetilde{g}^{\sigma}}(\lambda,\xi)\cos(\lambda t)\\
&-\lambda \overline{\widetilde{f}^{\sigma}}(\lambda,\xi)\sin(\lambda t)\cdot\widetilde{g}^{\sigma}(\lambda,\xi)\cos(\lambda t).
\end{align*}
and
\begin{align*}
(\ref{eq:energi2})=&\lvert \lambda \widetilde{f}^{\sigma}(\lambda,\xi)\cos(\lambda t)+\widetilde{g}^{\sigma}(\lambda,\xi)\sin(\lambda t)\rvert^2\\
=&\lambda^2 \lvert \widetilde{f}^{\sigma}(\lambda,\xi)\rvert^2 \cos^2(\lambda t)+\lvert \widetilde{g}^{\sigma}(\lambda,\xi)\rvert^2\sin^2(\lambda t)\\
&+\lambda \widetilde{f}^{\sigma}(\lambda,\xi)\cos(\lambda t)\cdot\overline{\widetilde{g}^{\sigma}}(\lambda,\xi)\sin(\lambda t)\\
&+\lambda \overline{\widetilde{f}^{\sigma}}(\lambda,\xi)\cos(\lambda t)\cdot\widetilde{g}^{\sigma}(\lambda,\xi)\sin(\lambda t).
\end{align*}
Hence we obtain:
\begin{align*}
(\ref{eq:energi1})+(\ref{eq:energi2})=&\lambda^2 \lvert\widetilde{f}^{\sigma}(\lambda,\xi)\rvert^2\sin^2(\lambda t)+\lvert\widetilde{g}^{\sigma}(\lambda,\xi)\rvert^2\cos^2(\lambda t)\\
&+\lambda^2\lvert  \widetilde{f}^{\sigma}(\lambda,\xi)\rvert^2 \cos^2(\lambda t)+\lvert \widetilde{g}^{\sigma}(\lambda,\xi)\rvert^2\sin^2(\lambda t)\\
=&\lambda^2\lvert  \widetilde{f}^{\sigma}(\lambda,\xi)\rvert^2+\lvert\widetilde{g}^{\sigma}(\lambda,\xi)\rvert^2.
\end{align*}
Therefore the total energy
is given by 
\begin{align*}
2\mathcal{E}(\varphi)(t)=&\lVert \lambda \tilde{f}^{\sigma}\rVert^2_{L^2((0,\infty)\times \partial X,C_0 \vert \mathbf{c}(\lambda)\vert^{-2} \,d\mu_{\sigma}(\xi)\,d\lambda)}\\
&+\lVert \tilde{g}^{\sigma}\rVert^2_{L^2((0,\infty)\times \partial X,C_0 \vert \mathbf{c}(\lambda)\vert^{-2} \,d\mu_{\sigma}(\xi)\,d\lambda)}
\end{align*}
and is independent of the time. 
\end{proof}
Note that using a different method one can proof the conservation of energy of solutions of the shifted wave equation on an arbitrary oriented Riemannian manifolds (see \cite{helgason1994geometric}{CH.V Lemma 5.12}). But via this proof one does not obtain the explicit expression for the total energy above. 
Using Theorem \ref{thm:energy}, Greens identity and the fact that $f$ has compact support we obtain that:
\begin{align*}
2\mathcal{E}(\varphi)=\lVert g \rVert^2_{L^2(X)}+\lVert \nabla f \rVert^2_{L^2(X)}-\rho^2\lVert f \rVert^2_{L^2(X)}.
\end{align*}
Hence comparing the above with the expression for the energy from Theorem \ref{thm:energy} we obtain using the Plancherel theorem and Lemma \ref{lemma:fourierdelta}
\begin{align}
\lVert \nabla f \rVert^2_{L^2(X)}&-\rho^2\lVert f \rVert^2_{L^2(X)}\label{eq:gradient}\\
&=\lVert \lambda \tilde{f}^{\sigma}\rVert^2_{L^2((0,\infty)\times \partial X,C_0 \vert \mathbf{c}(\lambda)\vert^{-2} \,d\mu_{\sigma}(\xi)\,d\lambda)}.\nonumber
\end{align}
In the next section we are going to investigate the term on the right hand side to obtain bounds on the energy just using the $L^2$ norms of the initial conditions. 

\section{A Paley-Wiener Type Theorem on Harmonic Manifolds of Rank One}
The classical  Paley-Wiener theorem (see for instance \cite[p.161]{Yosida_1974}) gives shape bounds on the decay of the Fourier transform of a compactly supported function on $\R^n$:
\begin{satz} \label{thm:classPW}
A holomorphic function $F:\C^n\to\C$ is the Fourier transform of a smooth function with support in the ball $\{x\in \R^n\mid \lVert x\rVert\leq R\}$ if and only if for every $N\in \N_{>0}$ there exists a constant $C_N>0$ such that 
$$ \lvert F(\lambda)\rvert\leq C_N(1+\lvert \lambda \rvert )^{-N}e^{R\lvert\operatorname{Im} \lambda\rvert}\quad \forall \lambda \in \C.$$
\end{satz}
In this section we want to show a weaker statement (Theorem \ref{thm:PWschwach}) namely that a sufficient decay of the derivatives of a function forces there Fourier transform to have support within a bounded set. 
Using mainly Lemma \ref{lemma:fourierdelta} and the Plancherel theorem this is an extension of a Paley-Wiener type theorem from \cite{AMBP_2010__17_2_327_0} to harmonic manifolds of rank one.
The proof follows the lines in \cite{AMBP_2010__17_2_327_0} closely with the addition of some details, but the statement of the Paley-Wiener type theorem is weaker then the one in \cite{AMBP_2010__17_2_327_0} since it is still not known if the Fourier transform on harmonic manifolds is surjective. 
Furthermore we use this result to show that the total energy of a solution to the shifted wave equation with specific initial conditions is bounded by bounds only depending on the $L^2$ norm of the initial conditions and bounds on the support of the Fourier transform of the initial conditions. 
Let $g:\R^+\times\partial X\to \C$ be a measurable function with respect to the measure $C_0 \vert \mathbf{c}(\lambda)\vert^{-2} \,d\mu_{\sigma}(\xi)\,d\lambda$ then we define 
$$R_g:=\sup_{(\lambda,\xi)\in\operatorname{supp }g}\lvert\lambda\rvert.$$ Note that this might be infinite. 
\begin{lemma}\label{lemma:PW}
Let $g$ be a function on $\R^+\times\partial X$ such that $(\lambda,\xi)\to\lambda^jg(\lambda,\xi)$ belongs to $L^2(\R^+\times\partial X,C_0 \vert \mathbf{c}(\lambda)\vert^{-2} \,d\mu_{\sigma}(\xi)\,d\lambda)$ for all integers $j$. Then 
$$ R_g=\lim_{j\to\infty}\Big (C_0 \int_0^{\infty}\int_{\partial X} \lambda^{2j}\lvert g(\lambda,\xi)\rvert^2  \vert \mathbf{c}(\lambda)\vert^{-2} \,d\mu_{\sigma}(\xi)\,d\lambda\Big)^{1/(2j)} $$
\end{lemma}
\begin{proof}
First we assume $R_g<\infty$ then let $0<\epsilon<R_g$ and we get for some $\delta >0$ that:
\begin{align*}
C_0 \int_{0}^{R_g-\epsilon}\int_{\partial X} \lambda^{2j}\lvert g(\lambda,\xi)\rvert^2\vert \mathbf{c}(\lambda)\vert^{-2} \,d\mu_{\sigma}(\xi)\,d\lambda\geq (R_g-\epsilon)^{2j+1}\delta.
\end{align*}
Hence we have:
\begin{align*}
\lim\inf_{j\to\infty}&\Big (C_0\int_0^{\infty}\int_{\partial X} \lambda^{2j}\lvert g(\lambda,\xi)\rvert^2 \vert \mathbf{c}(\lambda)\vert^{-2} \,d\mu_{\sigma}(\xi)\,d\lambda\Big)^{1/(2j)}\\
&\geq \lim\inf_{j\to\infty}\Big (C_0 \int_{0}^{R_g-\epsilon}\int_{\partial X} \lambda^{2j}\lvert g(\lambda,\xi)\rvert^2\vert \mathbf{c}(\lambda)\vert^{-2} \,d\mu_{\sigma}(\xi)\,d\lambda\Big)^{1/(2j)}\\
&\geq R_g-\epsilon.
\end{align*}
On the other hand: 
\begin{align*}
\lim\sup_{j\to\infty}&\Big (C_0\int_0^{\infty}\int_{\partial X} \lambda^{2j}\lvert g(\lambda,\xi)\rvert^2 \vert \mathbf{c}(\lambda)\vert^{-2} \,d\mu_{\sigma}(\xi)\,d\lambda\Big)^{1/(2j)}\\
&\leq R_g\lim\sup_{j\to\infty}\lVert g\rVert^{1/j}_{L^2(\R^+\times \partial X,C_0 \vert \mathbf{c}(\lambda)\vert^{-2} \,d\mu_{\sigma}(\xi)\,d\lambda)}\\
&=R_g.
\end{align*}
Since $\epsilon>0$ is arbitrary this completes the case $R_g<\infty$. 
Now suppose $R_g=\infty$. Then for every $M>0$ we have:
\begin{align*}
C_0 \int_M^{\infty} \int_{\partial X} \lambda^{2j}\lvert g(\lambda,\xi)\rvert^2\vert \mathbf{c}(\lambda)\vert^{-2} \,d\mu_{\sigma}(\xi)\,d\lambda>0\\
\end{align*}
and 
\begin{align*}
\lim\inf_{j\to\infty}&\Big (C_0\int_0^{\infty}\int_{\partial X} \lambda^{2j}\lvert g(\lambda,\xi)\rvert^2 \vert \mathbf{c}(\lambda)\vert^{-2} \,d\mu_{\sigma}(\xi)\,d\lambda\Big)^{1/(2j)}\\
&\geq \lim\inf_{j\to\infty}\Big(C_0\int_M^{\infty}\int_{\partial X} \lambda^{2j}\lvert g(\lambda,\xi)\rvert^2C_0 \vert \mathbf{c}(\lambda)\vert^{-2} \,d\mu_{\sigma}(\xi)\,d\lambda\Big)^{1/(2j)}\\
&\geq M.
\end{align*}
\end{proof}
\begin{defi}
Let $R>0$. We define:
 \begin{align*}
& L^2_R(\R^+\times\partial X,C_0 \vert \mathbf{c}(\lambda)\vert^{-2} \,d\mu_{\sigma}(\xi)\,d\lambda)\\
 &:=\{g\in   L^2(\R^+\times\partial X,C_0 \vert \mathbf{c}(\lambda)\vert^{-2} \,d\mu_{\sigma}(\xi)\,d\lambda)
 \mid R_g=R\}
\end{align*}
and
\begin{align*}PW^2_R(X):=\{f\in C^{\infty}(X)\mid &\Delta^jf\in L^2(X)\,\forall j\in \N \\
&\text{and} \lim_{j\to\infty}\lVert (\Delta +\rho^2)^jf\rVert_2^{1/(2j)}=R\}.
\end{align*}

\end{defi}
\begin{satz}\label{thm:PWschwach}
Let $R>0$ then, if it exists, the inverse Fourier transform of a function in $L^2_R(\R^+\times\partial X,C_0 \vert \mathbf{c}(\lambda)\vert^{-2} \,d\mu_{\sigma}(\xi)\,d\lambda)$ belongs to $PW^2_R(X)$ and the Fourier transform maps $PW^2_R(X)$to  $$L^2_R(\R^+\times\partial X,C_0 \vert \mathbf{c}(\lambda)\vert^{-2} \,d\mu_{\sigma}(\xi)\,d\lambda).$$ 
\end{satz}
\begin{proof}
Let $g\in  L^2_R(\R^+\times\partial X,C_0 \vert \mathbf{c}(\lambda)\vert^{-2} \,d\mu_{\sigma}(\xi)\,d\lambda)$ and denote its inverse Fourier transformed with respect to $\sigma\in X$ by $f$. $f$ is smooth by the Lebesgue's dominant convergence theorem and $f$ satisfies condition $(1)$ since by Lemma \ref{lemma:fourierdelta} we have:
\begin{align*}
\Delta^jf=(-1)^jC_0 \int_0^{\infty}\int_{\partial X} &(\lambda^2+\rho^2)^j \tilde{f}^{\sigma}(\lambda,\xi)\\\cdot &e^{(i\lambda-\rho)B_{\xi,\sigma}(x)}\vert \mathbf{c}(\lambda)\vert^{-2} \,d\mu_{\sigma}(\xi)\,d\lambda
\end{align*}
and $\tilde{f}^{\sigma}\in L^2_R(\R^+\times\partial X,C_0 \vert \mathbf{c}(\lambda)\vert^{-2} \,d\mu_{\sigma}(\xi)\,d\lambda)$.
 Using the Plancherel  theorem, Lemma \ref{lemma:fourierdelta} and Lemma \ref{lemma:PW} we have:
 %%%%%%%%%%%
\begin{align*}
\lim_{j\to\infty}\lVert &(\Delta +\rho^2)^jf\rVert_2^{1/(2j)}&\\
&=\lim_{j\to\infty}\Big (C_0 \int_0^{\infty}\int_{\partial X} \lambda^{2j}\lvert \tilde{f}^{\sigma}(\lambda,\xi)\rvert^2\vert \mathbf{c}(\lambda)\vert^{-2} \,d\mu_{\sigma}(\xi)\,d\lambda\Big)^{1/(2j)}\\
&=\lim_{j\to\infty}\Big (C_0 \int_0^{\infty}\int_{\partial X} \lambda^{2j}\lvert g (\lambda,\xi)\rvert^2\vert \mathbf{c}(\lambda)\vert^{-2} \,d\mu_{\sigma}(\xi)\,d\lambda\Big)^{1/(2j)}\\
&=R.
\end{align*}
Now if $f\in PW^2_R(X)$, then by the Plancherel theorem and Lemma \ref{lemma:fourierdelta} we have:
$\Delta^{2j}\tilde{f}^{\sigma}$ is in  $L^2_R(\R^+\times\partial X,C_0 \vert \mathbf{c}(\lambda)\vert^{-2} \,d\mu_{\sigma}(\xi)\,d\lambda)$ and by Lemma \ref{lemma:PW} we have $R_g=R$.
\end{proof}
\begin{folg}
Let $\sigma\in X$ and $R>0$ then for a smooth solution of the shifted wave equation $\varphi:X\times \R\to\C$  with initial conditions
\begin{align*}
\varphi(x,0)&=f(x)\in PW^2_R(X)\\
\left. \frac{\partial}{\partial t}\right|_{t=0} \varphi(x,t)&=g(x)\in C^{\infty}_c(X)
\end{align*}
we have
\begin{align*}
2\mathcal{E}(\varphi)(t)\leq R^2\lVert f\rVert^2_{L^2(X)}+\lVert g\rVert^2_{L^2(X)}.
\end{align*}
Furthermore we obtain:
\begin{align*}
\lVert \nabla f\rVert^2_{L^2(X)}\leq (R^2+\rho^2)\lVert f\rVert^2_{L^2(X)}.
\end{align*}
\end{folg} 
\begin{proof}
 We have by Theorem \ref{thm:energy} that
\begin{align*}
2\mathcal{E}(\varphi)(t)=&\lVert \lambda \tilde{f}^{\sigma}\rVert^2_{L^2((0,\infty)\times \partial X,C_0 \vert \mathbf{c}(\lambda)\vert^{-2} \,d\mu_{\sigma}(\xi)\,d\lambda)}\\
&+\lVert \tilde{g}^{\sigma}\rVert^2_{L^2((0,\infty)\times \partial X,C_0 \vert \mathbf{c}(\lambda)\vert^{-2} \,d\mu_{\sigma}(\xi)\,d\lambda)}
\end{align*}
and since $f\in PW^2_R(X)$ we obtain:
\begin{align}\label{eq:gradient2}
\lVert \lambda \tilde{f}^{\sigma}\rVert^2&_{L^2((0,\infty)\times \partial X,C_0 \vert \mathbf{c}(\lambda)\vert^{-2} \,d\mu_{\sigma}(\xi)\,d\lambda)}\\ &\leq R^2 \lVert \tilde{f}^{\sigma}\rVert^2_{L^2((0,\infty)\times \partial X,C_0 \vert \mathbf{c}(\lambda)\vert^{-2} \,d\mu_{\sigma}(\xi)\,d\lambda)}.\nonumber
\end{align}
 Therefore applying the Plancherel theorem yields: 
\begin{align*}
2\mathcal{E}(\varphi)(t)\leq R^2\lVert f\rVert^2_{L^2(X)}+\lVert g\rVert^2_{L^2(X)}.
\end{align*}
Now using equation (\ref{eq:gradient}), equation (\ref{eq:gradient2}) and the Plancherel theorem we conclude: 
\begin{align*}
\lVert \nabla f\rVert^2_{L^2(X)}\overset{\text{(\ref{eq:gradient})}}&{=} \lVert \lambda \tilde{f}^{\sigma}\rVert^2_{L^2((0,\infty)\times \partial X,C_0 \vert \mathbf{c}(\lambda)\vert^{-2} \,d\mu_{\sigma}(\xi)\,d\lambda)}+\rho^2\lVert f \rVert^2_{L^2(X)}\\
\overset{\text{(\ref{eq:gradient2})}}&{\leq} R^2 \lVert  \tilde{f}^{\sigma}\rVert^2_{L^2((0,\infty)\times \partial X,C_0 \vert \mathbf{c}(\lambda)\vert^{-2} \,d\mu_{\sigma}(\xi)\,d\lambda)}+\rho^2\lVert f \rVert^2_{L^2(X)}\\
\overset{\text{Plancherel theorem}}&{=} R^2\lVert f\rVert^2_{L^2(X)}+\rho^2\lVert f \rVert^2_{L^2(X)}\\
&=(R^2+\rho^2)\lVert f\rVert^2_{L^2(X)}.
\end{align*}
\end{proof}

\section{The Paley Wiener Theorem for Harmonic Manifolds of Rank One}
\begin{satz}\label{thm:PW2}
Let $f:X\to \C$ be a smooth function with compact support in the ball $B(\sigma,R)$ for some $\sigma\in X$ and $R>0$ then the Fourier transform of $f$ based at $\sigma$
\begin{align*}
\tilde{f}^{\sigma}(\lambda,\xi)=\int_X f(x)e^{(-i\lambda-\rho)B_{\xi,\sigma}(x)}\,dx
\end{align*}
is a holomorphic function in $\lambda$ and we have:
\begin{align*}
\sup_{\lambda\in \C,\,\xi\in\partial X} e^{-R\lvert\operatorname{Im}( \lambda) \rvert} (1+\lvert \lambda\rvert)^{N} \lvert \tilde{f}^{\sigma}(\lambda,\xi)\rvert<\infty \quad \forall N\in \N_{>0}.
\end{align*}
\end{satz}
The above is a generalisation of theorem 4.5 in \cite{astengocamporesiblasio1997} but our method differs from theirs which relies on the homogeneity of Damek-Ricci spaces. Furthermore the boundary structure of the Damek-Ricci space $NA$ used consist of the non compact group $N$ wheres we use the geometric boundary which is equivalent to using the one point compactification of $N$, for an explanation of this correspondence see for example \cite[Section 3]{MR2474430}. 
The idea of the proof: We first show that for $f\in C^{\infty}_c(X)$ the Radon transform $\mathcal{R}_{\sigma}(f)(s,\xi)$, a modification of the one introduced in \cite{Rouvire2021}, is smooth in $s$. Then we argue that it vanishes for $s>R$ and all $\xi \in \partial X$. Using the connection of the Radon transform and the Fourier transform via the Euclidean Fourier transform we apply the classical Paley-Wiener theorem to show the claim. This approach is also used by Helgason to show the Paley Wiener theorem for non compact symmetric space (see \cite[p.278]{helgason1994geometric}).
We begin by introducing the Radon transform, a generalisation of the Abel transform to non radial functions. 
\subsection{The Radon transform}

We define the Radon transform  $\mathcal{R}_{\sigma}(f):\R\times \partial X \to \C $ at $\sigma\in X$ for $f\in C^{\infty}_c(X)$
 by:
$$\mathcal{R}_{\sigma}(f)(s,\xi):=e^{-\rho s}\int_{H_{\xi,\sigma}(s)}f(z)\,dz$$
for all $s\in \R$ and $\xi\in \partial X$.
Note that this definition differs from the one given in  \cite{Rouvire2021} by the factor $e^{-\rho s}$, furthermore all signs are swapped compared to his work since he chooses the Busemann function to be defined with the opposite sign to ours. We choose this factor deliberately to have a direct correspondence to the Fourier transform via the Euclidean Fourier transform in Lemma \ref{lemma:PW2} and obtain the Abel transform on radial functions.
\begin{lemma}\label{lemma:PW1}
Let $f\in C^{\infty}_c(X)$ then $\mathcal{R}_{\sigma}(f)(s,\xi)$ is smooth in $s$. 
\end{lemma}
\begin{proof}
In coordinates given by the diffeomorphisem (\ref{eq:diffeo1}) and by  (\ref{eq:diffeo2}) the regularity of $\mathcal{R}_{\sigma}(f)(s,\xi)$  in $s$ is given by the minimum of the regularity of $f$ and $\Psi_s$. But since the Busemann functions and the metric are analytic $\Psi_s$ is analytic in $s$. Hence $\mathcal{R}_{\sigma}(f)(s,\xi)$ is smooth in $s$. 
\end{proof}
 The lemma is a version of the projection slice theorem for harmonic manifolds.
\begin{lemma}\label{lemma:PW3}
Let $f\in C^{\infty}_c(X)$ have support in the ball $B(\sigma,R)$ for some $\sigma\in X$ and $R>0$ then $\mathcal{R}_{\sigma}(f)(s,\xi)=0$ for $\lvert s\rvert >R$ and all $\xi\in \partial X$. 
\end{lemma}
\begin{proof}
Let $\lvert s\rvert >R$.
Since the Busemann function is Lipschitz with Lipschitz constant $1$ we have that 
$\lvert B_{\xi,\sigma}(x)\rvert $ is a lower bound of $d(\sigma,x)$. Hence for all $x\in H^{s}_{\xi,\sigma}$ we have that $d(\sigma,x)>R$ hence 
$f=0$ on $H^{s}_{\xi,\sigma}$ and therefore $$\mathcal{R}_{\sigma}(f)(s,\xi)=e^{-\rho s}\int_{H_{\xi,\sigma}(s)}f(z)\,dz=0$$
 for all $\xi\in \partial X.$ 
\end{proof}
\begin{bem}
Since the gradient of the Busemann function $B_{\xi,\sigma}$ in $\sigma\in X$ coincides up to a sing with the initial condition of the unique geodesic emitting from $\sigma$ and ending in $\xi$ the distance from $H^{s}_{\xi,\sigma}$ is given by $\lvert s\rvert$. 
\end{bem}
In the next lemma the choice of the factor $e^{-\rho s}$ will become apparent. A version without the factor can be found in \cite[Proposition 9]{Rouvire2021}.
\begin{lemma}\label{lemma:PW2}
Let $\mathcal{F}$ be the Euclidean Fourier transform  given for a smooth complex valued function $u$ on $\R$  with compact support by
$$\mathcal{F}(u)(\lambda)=\int_{-\infty}^{\infty} e^{-i\lambda t}u(t)\,dt\quad \lambda \in \C,$$
 then for $f\in C^{\infty}_c(X)$ we have:
$$ \tilde{f}^{\sigma}(\lambda,\xi)=\mathcal{F}\big(\mathcal{R}_{\sigma}(f)(\cdot,\xi)\big)(\lambda).$$
\end{lemma}
\begin{proof}
We have for $f\in C^{\infty}_c(X)$ using the Co-area formula:
\begin{align*}
\tilde{f}^{\sigma}(\lambda,\xi)&=\int_{X}f(x)e^{-(i\lambda+p)B_{\xi,\sigma}(x)}\,dx\\
&=\int_{-\infty}^{\infty}\int_{H_{s,\xi}}f(z)e^{-(i\lambda+p)s}\,dz\,ds\\
&=\int_{-\infty}^{\infty}e^{-i\lambda s}e^{-ps}\int_{H_{s,\xi}}f(z)\,dz,ds\\
&=\int_{-\infty}^{\infty}e^{-i\lambda s} \mathcal{R}_{\sigma}(f)(s,\xi)\,ds\\
&=\mathcal{F}(\mathcal{R}_{\sigma}(f)(s,\xi))(\lambda).
\end{align*}
Where we get the existence of the Euclidean Fourier transform above from Lemma \ref{lemma:PW3}.
\end{proof}
\begin{bem}
In \cite[Theorem 11]{Rouvire2021} Rouvi\`ere uses Lemma \ref{lemma:PW2} to prove a inversion formula for the Radon transform. The idea is to apply the inverse Fourier transform on $X$ to the the result of the lemma. 
\end{bem}

\begin{proof}[Proof of Theorem \ref{thm:PW2}]
First we note that $e^{(-i\lambda-\rho)B_{\xi,\sigma}(x)}$ is for all $x\in X$ holomorphic in $\lambda \in \C$ and since
\begin{align*}
\tilde{f}^{\sigma}(\lambda,\xi)=\int_X f(x)e^{(-i\lambda-\rho)B_{\xi,\sigma}(x)}\,dx,
\end{align*}
it is sufficient to show that 
\begin{align*}
\int_X \left\lvert f(x)e^{(-i\lambda-\rho)B_{\xi,\sigma}(x)}\right\rvert\,dx<\infty\quad\forall \lambda\in \C.
\end{align*}
But this is given by the fact that $f$ has compact support. Hence $\tilde{f}^{\sigma}(\lambda,\xi)$
is holomorphic in $\lambda\in \C$ for all $\xi\in \partial X$ by Morera's theorem. Now by Lemma \ref{lemma:PW1} $\mathcal{R}_{\sigma}(f)(s,\xi)$ is smooth in $s$ and by Lemma \ref{lemma:PW3} $\mathcal{R}_{\sigma}(f)(s,\xi)$ has support in $[-R,R]$. Furthermore by Lemma \ref{lemma:PW2} $$\tilde{f}^{\sigma}(\lambda,\xi)=\mathcal{F}\big(\mathcal{R}_{\sigma}(f)(s,\xi)\big)(\lambda).$$
 Hence by the classical Paley-Wieder theorem (see Theorem \ref{thm:classPW}) we have that for every $\xi\in \partial X$ and $N\in \N_{>0}$ there exists a constant $C_{N,\xi}>0$ such that
$$ \lvert  \tilde{f}^{\sigma}(\lambda,\xi)\rvert \leq C_{N,\xi} (1+\lvert \lambda \rvert)^{-N}e^{R\lvert \operatorname{Im} \lambda\rvert}\quad\forall \lambda\in \C.$$
Now $\partial X$ is compact and $\tilde{f}^{\sigma}(\lambda,\xi)$ is continuous in $\xi$, since the Busemann boundary and the geometric boundary coincide, hence there exists a $C_N>0$ such that for all $\xi\in \partial X$:
$$ \lvert  \tilde{f}^{\sigma}(\lambda,\xi)\rvert \leq C_{N} (1+\lvert \lambda \rvert)^{-N}e^{R\lvert \operatorname{Im} \lambda\rvert}\quad \forall\lambda\in\C.$$
This yields the claim.
\end{proof}
\begin{prop}\label{lemma:even}
Let $f\in C_c^{\infty}(X)$ then we have: 
\begin{align*}
\int_{\partial X} \tilde{f}^{\sigma}(-\lambda,\xi) e^{(-i\lambda -\rho)B_{\xi,\sigma}(x)}\,d\mu_{\sigma}(\xi)=\int_{\partial X} \tilde{f}^{\sigma}(\lambda,\xi) e^{(i\lambda -\rho)B_{\xi,\sigma}(x)}\,d\mu_{\sigma}(\xi).
\end{align*}
\end{prop}
The proof follows from the following lemma with the relation $$\varphi_{-\lambda,\sigma}=\varphi_{\lambda,\sigma}.$$

\begin{lemma}
Let $f\in C_c^{\infty}(X)$ then we have: 
\begin{align*}
f*\varphi_{\lambda,\sigma}(x):&=\int_X f(y)\cdot\varphi_{\lambda,x}(y)\,dy\\
&=\int_{\partial X} \tilde{f}^{\sigma}(-\lambda,\xi)\cdot e^{(-i\lambda -\rho)B_{\xi,\sigma}(x)}\,d\mu_{\sigma}(\xi).
\end{align*}
\end{lemma}
\begin{proof}
Recall the relations (\ref{coBuse}), (\ref{eq:RNdiv}), (\ref{eq:fourpoint}) and (\ref{radialeigen}). Then we obtain for $x,\sigma \in X$:

\begin{align*}
f*\varphi_{\lambda,\sigma}(x)&=\int_X f(y)\cdot\varphi_{\lambda,x}(y)\,dy\\
&\overset{\text{(\ref{radialeigen})}}{=}\int_X f(y)\cdot \int_{\partial X} e^{(i\lambda -\rho)B_{\xi,x}(y)}\,d\mu_x(\xi)\,dy\\
&=\int_X \int_{\partial X} f(y)e^{(i\lambda -\rho)B_{\xi,x}(y)}\,d\mu_x(\xi)\,dy\\
&= \int_{\partial X}\int_X f(y)e^{(i\lambda -\rho)B_{\xi,x}(y)}\,dy\,d\mu_x(\xi)\\
&=\int_{\partial X}\tilde{f}^{x}(-\lambda,\xi)\,d\mu_{x}(\xi)\\
&\overset{\text{(\ref{eq:fourpoint})}}{=}\int_{\partial X}\tilde{f}^{\sigma}(-\lambda,\xi)\cdot e^{(-i\lambda+\rho)B_{\xi,\sigma}(x)}   \,d\mu_{x}(\xi)\\
&\overset{\text{(\ref{eq:RNdiv})}}{=}\int_{\partial X}\tilde{f}^{\sigma}(-\lambda,\xi)\cdot e^{(-i\lambda+\rho)B_{\xi,\sigma}(x)} e^{-2\rho B_{\xi,\sigma}(x)}  \,d\mu_{\sigma}(\xi)\\
&\overset{\text{(\ref{coBuse})}}{=}\int_{\partial X} \tilde{f}^{\sigma}(-\lambda,\xi)\cdot e^{(-i\lambda -\rho)B_{\xi,\sigma}(x)}\,d\mu_{\sigma}(\xi).
\end{align*}
The interchange of integrals is justified by the Fubini-Tonelli  theorem and the facts that $f$ has compact support and $\partial X$ has finite measure ($d\mu_{\sigma}(\xi)$ is a probability measure). 
\end{proof}
\begin{folg}\label{folg:PW}
Let $R>0$ and denote by $PW^0_R$ all functions $F:\C\times \partial X\to \C$ holomorphic on $\C$ which satisfy 
\begin{align*}
\sup_{\lambda\in \C,\,\xi\in\partial X} e^{-R\lvert\operatorname{Im}( \lambda) \rvert} (1+\lvert \lambda\rvert)^{N} \lvert F(\lambda,\xi)\rvert<\infty \quad \forall N\in \N_{>0}.
\end{align*}
and for $\sigma\in X$:
$$\int_{\partial X} F(-\lambda,\xi)\cdot e^{(-i\lambda -\rho)B_{\xi,\sigma}(x)}\,d\mu_{\sigma}(\xi)=\int_{\partial X} F(\lambda,\xi)\cdot e^{(i\lambda -\rho)B_{\xi,\sigma}(x)}\,d\mu_{\sigma}(\xi),$$
then the image of $C^{\infty}_c(X)$ under the Fourier transform based at $\sigma$ is contained in 
$$\bigcup_{R\geq 0} PW^0_R.$$
\end{folg}
\section{Huyghens' principle}
In this section we want to prove an asymptotic Huyghens' principle along the lines of the proof of \cite{Branson1995}. For this we need to make assumptions on the $\mathbf{c}$-function, namely we need that the function $\eta$ defined by $\eta(\lambda)^{-1}:=\mathbf{c}(\lambda)\overline{\mathbf{c}(\lambda)}$ on the lower have plane of $\C$ has a holomorphic extension up to $\operatorname{Im}(\lambda)=\epsilon_{max}>0$ where it has a singular pole and is a polynomial with real coefficients up to this point such that $\eta(\lambda)=\lambda^{n-1}\eta_0(\lambda)$ where all poles of $\eta$ are also poles of $\eta_0$ with the same multiplicity. This condition is satisfied in the case of symmetric spaces of rank one and Damek-Ricci spaces whose nilpotent part has a centre of even dimension as well as on the hyperbolic spaces of odd dimension. For this see \cite{AMBP_2005__12_1_147_0}. For more detail on the $\mathbf{c}$-function of Damek-Ricci space see \cite{varadarajan2006harmonic}, especially proposition 4.7.13-4.7.15 and theorem 6.3.4.
\begin{bem}\label{bem:ceven}
Note that $\eta(\lambda)=\lvert \mathbf{c}(\lambda)\rvert^{-2}$ and that by \cite[Lemma 3.4 and Proposition 3.17]{Bloom1995TheHM} (alternatively one can observe this from (\ref{eq:eigenlaplace}) combined with (\ref{eq:symeigen}) and  (\ref{eigende})) we have:
\begin{align*}
\mathbf{c}(\lambda)=\overline{\mathbf{c}(-\lambda)}\quad\forall \lambda\in \R.
\end{align*}
From this we get that for all $\lambda\in \R$
\begin{align*}
\eta(-\lambda)=(\mathbf{c}(-\lambda)\overline{\mathbf{c}(-\lambda)})^{-1}=(\overline{\mathbf{c}(\lambda)}\mathbf{c}(\lambda))^{-1}=\eta(\lambda)
\end{align*}
hence $\eta$ is even in $\lambda$.
\end{bem}
\begin{satz}\label{thm:Huy}
Let $(X,g)$ be a non compact simply connected harmonic manifold of rank one of dimension bigger then one, such that the $\mathbf{c}$-function satisfies the condition above. And let $\varphi$ be a solution of the shifted wave equation with initial conditions $f,g$ supported in a ball of radius $R$ around $\sigma\in X$. Let $\epsilon_{max}$ be as above and $0<\epsilon<\epsilon_{max}<\infty$ then there is a constant $C>0$ such that
$$\lvert \varphi(x,t)\rvert \leq C(\epsilon_{max}-\epsilon)^{-1}\cdot e^{-\epsilon(\lvert t\rvert-d(x,\sigma)-R)}\quad \forall (x,t)\in X\times\R$$
and if $\epsilon_{max}=\infty$ we get
$$\lvert \varphi(x,t)\rvert \leq C\cdot e^{-\epsilon(\lvert t\rvert-d(x,\sigma)-R)}\quad\forall \epsilon>0, \forall (x,t)\in X\times\R$$
therefore we get:
$$\varphi(x,t)=0\quad\text{for } \lvert t \rvert -d(x,\sigma)\geq R.$$
\end{satz}
The proof of this statement will be conducted via a series of lemma occupying the remainder of the section. We will always require the  assumptions of the theorem. 
\begin{lemma}\label{lemma:contur}
Let $h:\C\to\C$ be a function holomorphic on the stripe $P=\{z\in \C\mid 0\leq\operatorname{Im} z \leq \epsilon\}$ such that there is a $C>0$ with $\lvert h(z)\rvert \leq C(1+\lvert z\rvert)^{-N}$ for some $N>0$ on $P$. Then:
\begin{align*}
\int_{-\infty}^{\infty}h(z)\,dz=\int_{-\infty}^{\infty} h(a+i\epsilon)\,da.
\end{align*}
\end{lemma}
\begin{proof}
Consider the contour in Figure \ref{fig2}.
Let $\gamma_1:[0,1]\to\C$ be given by $\gamma_1(s)=r+is\epsilon$ and $\gamma_2:[0,1]\to\C$ be given by $\gamma_2(s)=-r+i(1-s)\epsilon$ then by the bounds on $h$ on the stripe $P$ there are constants $C_1,C_2>0$ such that:
\begin{align*}
\Big\lvert \int_{\gamma_1}h\,ds\Big\rvert=\Big\lvert \int_0^1h(r+is\epsilon)\cdot i\theta\,ds\Big\rvert\leq C_1(1+\lvert r\rvert)^{-N}\\
\Big\lvert \int_{\gamma_2}h\,ds\Big\rvert=\Big\lvert \int_0^1h(-r+(1-is)\epsilon)\cdot -i\theta\,ds\Big\rvert\leq C_2(1+\lvert r\rvert)^{-N}.\\
\end{align*}
Therefore since both integrals tend to zero for $r\to \pm \infty$ and we get the assertion. 
\end{proof}
%%%%%%%%%%%%%%%%%%%%%%%%%%%%%%%%%
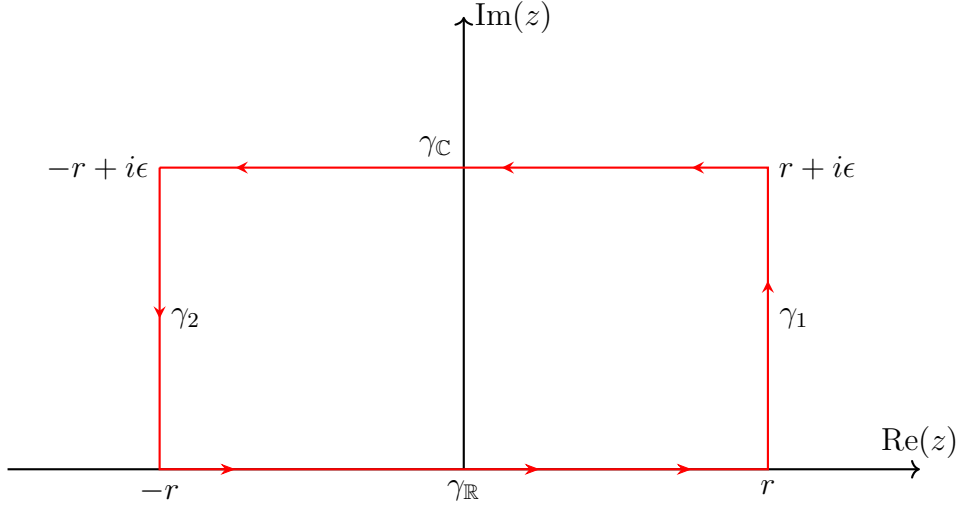
\begin{figure}[ht]
\caption{Contour of Lemma \ref{lemma:contur}, for $r\to\infty$ the integral along $\gamma_1$ and $\gamma_2$ vanishes because of the bounds on $h$.}\label{fig2}
\begin{tikzpicture}[decoration={markings,
    mark=at position 1cm   with {\arrowreversed[line width=1pt]{stealth}},
    mark=at position 4.5cm with {\arrowreversed[line width=1pt]{stealth}},
    mark=at position 7cm   with {\arrowreversed[line width=1pt]{stealth}},
    mark=at position 9.5cm with {\arrowreversed[line width=1pt]{stealth}},
    mark=at position 22cm   with {\arrowreversed[line width=1pt]{stealth}},
  }]
  \draw[thick, ->] (-6,0) -- (6,0) coordinate (xaxis);

  \draw[thick, ->] (0,0) -- (0,6) coordinate (yaxis);

  \node[above] at (xaxis) {$\mathrm{Re}(z)$};

  \node[right]  at (yaxis) {$\mathrm{Im}(z)$};
   \node [below] at (-4,0) {$-r$};
   \node [below] at (4,0) {$r$};
    \node [right] at (4,4) {$r+i\epsilon$};
    \node [left] at (-4,4) {$-r+i\epsilon$};
    \draw [red,-{Stealth[length=2mm, width=1.5mm]}](0,0) -- (1,0);
    \draw [red,-{Stealth[length=2mm, width=1.5mm]}](2,0) -- (3,0);
    \draw [red,-{Stealth[length=2mm, width=1.5mm]}](-4,0) -- (-3,0);
  \path[draw,red, line width=0.8pt, postaction=decorate] (-4,4)
    -- node[midway, above left, black] {$\gamma_{\C}$} (4,4)
    -- node[midway, right, black] {$\gamma_1$}(4,0)
    -- node[midway, below, black] {$\gamma_{\R}$} (-4,0)
    -- node[midway, right, black] {$\gamma_2$}(-4,4);
\end{tikzpicture}
\end{figure}

%%%%%%%%%%%%%%%%%%%%%%%%%%%%%%%%%%

\begin{lemma}\label{lemma:Huy1}
Let $f,g\in C^{\infty}_c(X)$ then the functions
\begin{align*}
F(\lambda,x)&:=\int_{\partial X} \tilde{f}^{\sigma}(\lambda,\xi)e^{(i\lambda -\rho)B_{\xi,\sigma}(x)}\eta(\lambda)\,d\mu_{\sigma}(\xi)\\
&\text{ and}\\
G(\lambda,x)&:=\int_{\partial X} \tilde{g}^{\sigma}(\lambda,\xi)e^{(i\lambda -\rho)B_{\xi,\sigma}(x)}\eta(\lambda)\,d\mu_{\sigma}(\xi)
\end{align*}
are even in $\lambda$ and 
\begin{align*}
\int_0^{\infty}&F(\lambda,x)\cos(\lambda t)+G(\lambda,\xi)\frac{\sin(\lambda t)}{\lambda}\,d\lambda\\
&=\frac{1}{2}\int_{-\infty}^{\infty}\big(F(\lambda,x)+\frac{G(\lambda,x)}{i\lambda}\big )e^{i\lambda t}\,d\lambda.
\end{align*}
\end{lemma}
%%%%%%
\begin{proof}
Since $\eta$ ,by Remark \ref{bem:ceven}, is even in $\lambda$ and by Proposition \ref{lemma:even} $F(\lambda,x)$ and $G(\lambda,x)$ are even in $\lambda$. 
Now using this and $$2\cos(\lambda t)=e^{i\lambda t}+e^{-i \lambda t}$$ we get:
\begin{align*}
\int_0^{\infty} F(\lambda ,x)\cos(\lambda t)\,d\lambda&=\frac{1}{2}\Big(\int_0^{\infty}F(\lambda,x)e^{i\lambda t}\,d\lambda +\int_0^{\infty}F(\lambda,x)e^{-i\lambda t}\,d\lambda\Big)\\
&=\frac{1}{2}\Big(\int_0^{\infty}F(\lambda,x)e^{i\lambda t}\,d\lambda +\int_{-\infty}^{0}F(\lambda,x)e^{i\lambda t}\,d\lambda\Big)\\
&=\frac{1}{2}\int_{-\infty}^{\infty}F(\lambda,x)e^{i\lambda t}\,d\lambda.
\end{align*}
Since $2i\sin(\lambda t)=e^{i\lambda t}-e^{-i\lambda t}$ and $G(\lambda,x)$ is even in $\lambda$ we obtain:
\begin{align*}
\int_0^{\infty}G(\lambda,x)\frac{\sin(\lambda t)}{\lambda}\,d\lambda=&\frac{1}{2i}\Big(\int_0^{\infty} G(\lambda,x)\frac{e^{i\lambda t}}{\lambda}d\lambda-\int_0^{\infty} G(\lambda,x)\frac{e^{-i\lambda t}}{\lambda}d\lambda\Big)\\
=&\frac{1}{2i}\Big(\int_0^{\infty} G(\lambda,x)\frac{e^{i\lambda t}}{\lambda}\,d\lambda+\int_{-\infty}^{0} G(\lambda,x)\frac{e^{i\lambda t}}{\lambda}\,d\lambda\Big)\\
=&\frac{1}{2}\int_{-\infty}^{\infty} G(\lambda,x)\frac{e^{i\lambda t}}{i\lambda}d\lambda.
\end{align*}
\end{proof}
By  \cite[Prop.6.1.1 and Prop. 6.1.4]{trimeche2018generalized} and (\ref{defi:ef}) we have the following bounds for the radial eigenfunctions of the Laplacian: 
\begin{lemma}\label{lemma:bounds}
For all $x,\sigma\in X$ and $\lambda\in \C$ we have:
\begin{enumerate}
\item$ \lvert \varphi_{\lambda,\sigma}(x)\rvert\leq\varphi_{i\operatorname{Im}(\lambda),\sigma}(x)\leq \varphi_{0,\sigma}(x)\cdot e^{\lvert \operatorname{Im}(\lambda)\rvert d(\sigma,x)}$,
\item $\lvert \operatorname{Im}(\lambda)\rvert \leq\rho \Rightarrow e^{(\lvert \operatorname{Im}(\lambda)\rvert -\rho)d(\sigma,x)}\leq \varphi_{i\operatorname{Im}(\lambda),\sigma}(x)\leq 1$,
\item $\lvert \operatorname{Im}(\lambda)\rvert \geq\rho \Rightarrow 1\leq \varphi_{i\operatorname{Im}(\lambda),\sigma}(x)\leq e^{(\lvert \operatorname{Im}(\lambda)\rvert -\rho)d(\sigma,x)}$.
\end{enumerate}
Furthermore, we have:
\begin{align*}
\varphi_{i\operatorname{Im}(\lambda),\sigma}(x)\leq k(1+d(\sigma,x))e^{(\lvert \operatorname{Im}(\lambda)\rvert -\rho)d(\sigma,x)}
\end{align*}
 for  some positive constant $k>0$. 
 \end{lemma}
\begin{lemma}\label{lemma:Huy2}
Assume the assumptions of the Theorem \ref{thm:Huy}.
Let $f,g\in C^{\infty}_c(X)$ with support in the ball of radius $R>0$ around $\sigma\in X$ then $F$ and $G$ admit holomorphic extensions in $\lambda$ up to $\epsilon_{\max}$ and for every $N\in\N$
 we can find a constant $C_N$ such that for all $\lambda\in \C$ with $0\leq\operatorname{Im}\lambda \leq \epsilon<\epsilon_{\max}$ and $x\in X$
\begin{align*}
\lvert F(\lambda,x)\rvert &\leq C_N (\epsilon_{max}-\epsilon)^{-1}(1+\lvert \lambda\rvert)^{-N} e^{\epsilon d(x,\sigma) +R \epsilon} \\
\end{align*}
and
\begin{align*}
\lvert G(\lambda,x)\rvert &\leq C_N (\epsilon_{max}-\epsilon)^{-1}(1+\lvert \lambda\rvert)^{-N} e^{\epsilon d(x,\sigma) +R\epsilon}.
\end{align*}
Furthermore if $\operatorname{dim} X > 1$ we have that for every $N\in\N$ there is a constant $D_N$ such that 
$$\lvert \lambda^{-1} G(\lambda,x)\rvert \leq D_N (\epsilon_{max}-\epsilon)^{-1}(1+\lvert \lambda\rvert)^{-N} e^{\epsilon d(x,\sigma) +R \epsilon }.$$
\end{lemma}
\begin{proof}
That $F,G$ are holomorphic up to $\epsilon_{max}$ in $\lambda$ follows from the fact that all functions making up those are holomorphic up to this point. 
Let us begin with the estimate on $F$ the one on $G$ follows in the same manner. 
\begin{align*}
\left\vert F(\lambda,x)\right\vert&\leq \Big\lvert\int_{\partial X} \tilde{f}^{\sigma}(\lambda,\xi)e^{(i\lambda -\rho)B_{\xi,\sigma}(x)}\eta(\lambda)\,d\mu_{\sigma}(\xi)\Big\rvert\\
&\leq \sup_{\operatorname{Im}\lambda <\epsilon_{\max},\,\xi\in\partial X}\lvert \tilde{f}^{\sigma}(\lambda,\xi)\eta(\lambda)\rvert \Big\lvert \int_{\partial X} e^{(i\lambda -\rho)B_{\xi,\sigma}(x)}\,d\mu_{\sigma}(\xi)\Big\rvert.\\
\end{align*}
By Lemma \ref{lemma:bounds} (1) and the integral representation of the radial eigenfunctions (\ref{radialeigen}):
\begin{align*}
 \Big\lvert \int_{\partial X} e^{(i\lambda -\rho)B_{\xi,\sigma}(x)}\,d\mu_{\sigma}(\xi)\Big\rvert&=\lvert\varphi_{\lambda,\sigma}(x)\rvert\\
& \leq \lvert\varphi_{i\operatorname{Im}\lambda}(x)\rvert\\
& \leq \lvert \varphi_{0,\sigma}(x)\rvert e^{\lvert \operatorname{Im}\lambda\rvert d(x,\sigma)}\\
& \leq e^{\lvert \operatorname{Im}\lambda\rvert d(x,\sigma)}.
\end{align*}
Now using Theorem \ref{thm:PW2}, the assumption that $\eta$ has a singular pole at $\epsilon_{max}$ and is a polynomial and since $\partial X$ is compact we can conclude that for every $N\in \N$ there is a constant $C_N$ such that for all $0\leq\operatorname{Im}\lambda \leq \epsilon<\epsilon_{\max}$
\begin{align*}
\lvert F(\lambda,x)\rvert &\leq C_N (\epsilon_{max}-\epsilon)^{-1}(1+\lvert \lambda\rvert)^{-N} e^{\epsilon d(x,\sigma) +R\lvert \operatorname{Im} \lambda\rvert}\\
&\leq C_N (\epsilon_{max}-\epsilon)^{-1}(1+\lvert \lambda\rvert)^{-N} e^{\epsilon d(x,\sigma) +R\epsilon}.
\end{align*}
For the last estimate on $\lvert \lambda^{-1} G(\lambda,x)\rvert$ one only need to consider  that $\eta(\lambda)=\lambda^{n-1}\eta_0(\lambda)$ where all poles of $\eta$ are also poles of $\eta_0$ with the same multiplicity. Hence one only need to exclude the case where $\operatorname{dim}X=1$. Then we get using the same lines as above:
\begin{align*}
\lvert  \lambda^{-1} G(\lambda,x)\rvert&\leq \Big\lvert\int_{\partial X}\lambda^{-1}  \tilde{g}^{\sigma}(\lambda,\xi)e^{(i\lambda -\rho)B_{\xi,\sigma}(x)}\eta(\lambda)\,d\mu_{\sigma}(\xi)\Big\rvert\\
&\leq \sup_{\operatorname{Im}\lambda <\epsilon_{\max},\,\xi\in\partial X}\lvert \lambda^{-1} \tilde{g}^{\sigma}(\lambda,\xi)\eta(\lambda)\rvert \Big\lvert \int_{\partial X} e^{(i\lambda -\rho)B_{\xi,\sigma}(x)}\,d\mu_{\sigma}(\xi)\Big\rvert\\
&\leq \sup_{\operatorname{Im}\lambda <\epsilon_{\max},\,\xi\in\partial X}\Big(\lvert \lambda^{n-2} \tilde{g}^{\sigma}(\lambda,\xi)\eta_0(\lambda)\rvert \\
&\quad\quad\quad\cdot\Big\lvert \int_{\partial X} e^{(i\lambda -\rho)B_{\xi,\sigma}(x)}\,d\mu_{\sigma}(\xi)\Big\rvert \Big)
\end{align*}
and then again use the estimate 
$$  \Big\lvert \int_{\partial X} e^{(i\lambda -\rho)B_{\xi,\sigma}(x)}\,d\mu_{\sigma}(\xi)\Big\rvert\leq e^{\lvert \operatorname{Im}\lambda\rvert d(x,\sigma)}.$$
Hence we obtain using the same arguments as above that for every $N\in\N$ there is a constant $D_N$ such that for $0\leq\operatorname{Im}\lambda \leq \epsilon<\epsilon_{\max}$
$$\lvert \lambda^{-1} G(\lambda,x)\rvert \leq D_N (\epsilon_{max}-\epsilon)^{-1}(1+\lvert \lambda\rvert)^{-N} e^{\epsilon d(x,\sigma) +R \epsilon }.$$
\end{proof}
%%%%%%
\begin{proof}[Proof Theorem \ref{thm:Huy}]
First we note that $u(x,-t)$ solves the shifted wave equation with initial conditions $f,-g$ hence we only need to consider the case $t\geq 0$.
Let $0<\epsilon<\epsilon_{max}$ then using Lemma \ref{lemma:contur} we can move the integral defining $u$ from $\R$ to $\R+i\epsilon$, hence:
\begin{align*}
2\lvert \varphi(x,t)\rvert &=\Big\lvert C_0\int_{-\infty}^{\infty}\big(F(\lambda,x)+\frac{G(\lambda,x}{i\lambda}\big) e^{i\lambda t}\,d\lambda\Big\rvert\\
&=\Big\lvert C_0 e^{-\epsilon t} \int_{-\infty}^{\infty}\big(F(a+i\epsilon,x)+\frac{G(a+i\epsilon)}{i(a+i\lambda)}\big ) e^{iat}\,d\lambda\Big\rvert,
\end{align*}
now using Lemma \ref{lemma:Huy2} we obtain for $N\in \N$ a constant $C_N>0$ such that:
$$2\lvert \varphi(x,t)\rvert \leq C_N(\epsilon_{max}-\epsilon)^{-1}e^{-\epsilon(t-d(x,\sigma))}e^{R\epsilon} \int_{-\infty}^{\infty}(1+\lvert \lambda \rvert )^{-N}\,d\lambda.$$
Since the last integral is bounded we obtain the claim. For the case that the $\mathbf{c}$-function is an entire function and a polynomial one notice that we can ignore the therm $(\epsilon_{max}-\epsilon)^{-1}$ in all the estimates which yields the assertion in this case. 
\end{proof}
\section{Equidistribution of Energy} 
Under the same assumptions on the $\mathbf{c}$-function as in the last section we now want to proof an asymptotic equidistribution of the energy between the kinetic and potential energy of a wave on $X$. 
\begin{satz}\label{thm:energy2}
Let $(X,g)$ be a non compact simply connected harmonic manifold of rank one, such that the $\mathbf{c}$-function satisfies the mentioned in the beginning of section 9. And let $\varphi$ be a solution of the shifted wave equation with smooth initial conditions $f,g$ compactly supported within a ball of radius $R$ around $\sigma\in X$. Let $\epsilon_{max}$ be as before and $0<\epsilon<\epsilon_{max}<\infty$ then there is a constant $C>0$ such that we have for the potential and kinetic energy $\mathcal{P}$ and $\mathcal{K}$
$$\lvert \mathcal{K}(\varphi)(t)-\mathcal{P}(\varphi)(t)\rvert\leq C (\epsilon_{max}-\epsilon)^{-1}(e^{-2\epsilon(\lvert t\rvert -R)})\quad\forall t\in \R$$
and if $\epsilon_{max}=\infty$ we have
$$\mathcal{K}(\varphi)(t)=\mathcal{P}(\varphi)(t)\quad \forall \lvert t\rvert \geq R.$$
\end{satz}
The proof is similar to the proof of Theorem \ref{thm:Huy}. Let us begin with calculating the different between the kinetic and potential energy.

\begin{lemma}\label{lemma:energy1}
Let $\varphi$ be a solution of the shifted wave equation with initial conditions $f,g\in C^{\infty}_c(X)$ then:
\begin{align*}
\frac{2}{C_0}\Big( \mathcal{K}(\varphi)(t)-\mathcal{P}(\varphi)(t)\Big )=&\int_0^{\infty}\int_{\partial X}\Big (\big(-\lambda^2 \tilde{f}^{\sigma}(\lambda,\xi)\overline{\tilde{f}^{\sigma}}(\lambda,\xi)\\
&+\tilde{g}^{\sigma}(\lambda,\xi)\overline{\tilde{g}^{\sigma}}(\lambda,\xi)\big)\cos(2\lambda t)\\
&-\big(\tilde{f}^{\sigma}(\lambda,\xi)\overline{\tilde{g}^{\sigma}}(\lambda,\xi)\\
&+\tilde{g}^{\sigma}(\lambda,\xi)\overline{\tilde{f}^{\sigma}}(\lambda,\xi)\big )\\&\cdot\lambda\sin(2\lambda t)\Big)
 d\mu_{\sigma}\lvert c(\lambda)\rvert^{-2}d\lambda.
\end{align*}
\end{lemma}
%%%%%%
\begin{proof}
From Lemma \ref{lemma:energy} keeping the same notation:
\begin{align*}
(1)-(2)=&\lambda^2 \tilde{f}^{\sigma}(\lambda,\xi)\overline{\tilde{f}^{\sigma}}(\lambda,\xi)\sin^2(\lambda t)\\
&+\tilde{g}^{\sigma}(\lambda,\xi)\overline{\tilde{g}^{\sigma}}(\lambda,\xi)\cos^2(\lambda t)\\
&-2\lambda  \tilde{f}^{\sigma}(\lambda,\xi)\overline{\tilde{g}^{\sigma}}(\lambda,\xi)\sin(\lambda t)\cos(\lambda t)\\
&-2\lambda  \tilde{g}^{\sigma}(\lambda,\xi)\overline{\tilde{f}^{\sigma}}(\lambda,\xi)\sin(\lambda t)\cos(\lambda t)\\
&-\lambda^2 \tilde{f}^{\sigma}(\lambda,\xi)\overline{\tilde{f}^{\sigma}}(\lambda,\xi)\cos^2(\lambda t)\\
&-\tilde{g}^{\sigma}(\lambda,\xi)\overline{\tilde{g}^{\sigma}}(\lambda,\xi)\sin^2(\lambda t).
\end{align*}
Now using $\sin(x)\cos(x)=\frac{1}{2}\sin(2x)$ we obtain:
\begin{align*}
=&-\lambda^2 \tilde{f}^{\sigma}(\lambda,\xi)\overline{\tilde{f}^{\sigma}}(\lambda,\xi)\big(\cos^2(\lambda t)-\sin^2(\lambda t)\big)\\
&+\tilde{g}^{\sigma}(\lambda,\xi)\overline{\tilde{g}^{\sigma}}(\lambda,\xi)\big(\cos^2(\lambda t)-\sin^2(\lambda t)\big)\\
&-\lambda \big(\tilde{f}^{\sigma}(\lambda,\xi)\overline{\tilde{g}^{\sigma}}(\lambda,\xi)
+\tilde{g}^{\sigma}(\lambda,\xi)\overline{\tilde{f}^{\sigma}}(\lambda,\xi)\big )\sin(2\lambda t).
\end{align*}
Finally the claim follows from $\cos^2(x)-\sin^2(x)=\cos(2x)$.
\end{proof}
%%%%%%%%%%%%%%%
For us to be able to use the same arguments as in section 9 the following lemma is essential.  
\begin{lemma}\label{lemma:lemma}
Let $h_1,h_2\in C_c^{\infty}(X)$ and $\sigma\in X$ then for all $\lambda\in \R$ and $\xi \in \partial X$: 
\begin{enumerate}
\item $ \overline{\widetilde{h}^{\sigma}_1}(\lambda,\xi)=\widetilde{\bar{h}}_1^{\sigma}(-\lambda,\xi)$.
\item We have\\ 
$\int_{\partial X} \widetilde{h}^{\sigma}_1(\lambda,\xi) \overline{\widetilde{h}^{\sigma}_2}(\lambda,\xi)\,d\mu_{\sigma}(\xi)=\int_{\partial X} \widetilde{h}^{\sigma}_1(-\lambda,\xi) \overline{\widetilde{h}^{\sigma}_2}(-\lambda,\xi)\,d\mu_{\sigma}(\xi).$

\end{enumerate}
\end{lemma}
%%%%%%%%%%%%%
\begin{proof}
For the first assertion we only need to look at the definition of the Fourier transform:
\begin{align*}
 \overline{\widetilde{h}^{\sigma}_1}(\lambda,\xi)&=\overline{\int_X h_1(x)e^{(-i\lambda-\rho)B_{\xi,\sigma}(x)}\,dx}\\
 &=\int_X \overline{h}_1(x)e^{(i\lambda-\rho)B_{\xi,\sigma}(x)}\,dx\\
 &=\widetilde{\bar{h}}_1^{\sigma}(-\lambda,\xi).
\end{align*}
The second assertion follows now from the first together with Proposition \ref{lemma:even}:
\begin{align*}
\int_{\partial X} \widetilde{h}^{\sigma}_1(\lambda,\xi) \overline{\widetilde{h}^{\sigma}_2}&(\lambda,\xi)\,d\mu_{\sigma}(\xi)\\
\overset{\text{Def.\ref{def:fourier}}}&{=}\int_{\partial X}\Big(\int_X h_1(x)e^{(-i\lambda-\rho)B_{\xi,\sigma}(x)}\,dx\Big) \overline{\widetilde{h}^{\sigma}_2}(\lambda,\xi)\,d\mu_{\sigma}(\xi)\\
&=\int_{\partial X}\int_X h_1(x) \overline{\widetilde{h}^{\sigma}_2}(\lambda,\xi)e^{(-i\lambda-\rho)B_{\xi,\sigma}(x)}\,dx\,d\mu_{\sigma}(\xi)\\
&=\int_X \int_{\partial X}h_1(x) \overline{\widetilde{h}^{\sigma}_2}(\lambda,\xi)e^{(-i\lambda-\rho)B_{\xi,\sigma}(x)}\,d\mu_{\sigma}(\xi)\,dx\\
&=\int_Xh_1(x) \int_{\partial X} \overline{\widetilde{h}^{\sigma}_2}(\lambda,\xi)e^{(-i\lambda-\rho)B_{\xi,\sigma}(x)}\,d\mu_{\sigma}(\xi)\,dx\\
\overset{\text{(Lemma \ref{lemma:lemma}(i)}}&{=}\int_X h_1(x) \int_{\partial X} \widetilde{\bar{h}}_2^{\sigma}(-\lambda,\xi)e^{(-i\lambda-\rho)B_{\xi,\sigma}(x)}\,d\mu_{\sigma}(\xi)\,dx\\
\overset{\text{Lemma \ref{lemma:even}}}&{=}\int_X h_1(x) \int_{\partial X} \widetilde{\bar{h}}_2^{\sigma}(\lambda,\xi)e^{(i\lambda-\rho)B_{\xi,\sigma}(x)}\,d\mu_{\sigma}(\xi)\,dx\\
&=\int_X \int_{\partial X}h_1(x) \widetilde{\bar{h}}_2^{\sigma}(\lambda,\xi)e^{(i\lambda-\rho)B_{\xi,\sigma}(x)}\,d\mu_{\sigma}(\xi)\,dx\\
&=\int_{\partial X}\int_X h_1(x) \widetilde{\bar{h}}_2^{\sigma}(\lambda,\xi)e^{(i\lambda-\rho)B_{\xi,\sigma}(x)}\,dx\,d\mu_{\sigma}(\xi)\\
&=\int_{\partial X} \widetilde{\bar{h}}_2^{\sigma}(\lambda,\xi)\int_X h_1(x)e^{(i\lambda-\rho)B_{\xi,\sigma}(x)}\,dx\,d\mu_{\sigma}(\xi)\\
\overset{\text{Def.\ref{def:fourier}}}&{=}\int_{\partial X} \widetilde{h}^{\sigma}_1(-\lambda,\xi)\widetilde{\bar{h}}_2^{\sigma}(\lambda,\xi)(\lambda,\xi)\,d\mu_{\sigma}(\xi)\\
\overset{\text{\ref{lemma:lemma}(i)}}&{=}\int_{\partial X} \widetilde{h}^{\sigma}_1(-\lambda,\xi) \overline{\widetilde{h}^{\sigma}_2}(-\lambda,\xi)\,d\mu_{\sigma}(\xi).
\end{align*}
Here the interchange of integrals is justified by the Fubini-Tonelli  theorem and the facts that $h_1$ and $h_2$ have compact support and $\partial X$ has finite measure ($d\mu_{\sigma}(\xi)$ is a probability measure). 
\end{proof}
%%%%%%%%%%%%%%%%%
\begin{lemma}\label{lemma:energy2}
Under the conditions of Theorem \ref{thm:energy2} define
\begin{align*}
A(\lambda)&:=\int_{\partial X}\big(-\lambda^2 \tilde{f}^{\sigma}(\lambda,\xi)\overline{\tilde{f}^{\sigma}}(\lambda,\xi)\\
&+\tilde{g}^{\sigma}(\lambda,\xi)\overline{\tilde{g}^{\sigma}}(\lambda,\xi)\big)\eta(\lambda)\,d\mu_{\sigma}(\xi)\\
\text{and}\\
B(\lambda)&:=\int_{\partial X} \big(\tilde{f}^{\sigma}(\lambda,\xi)\overline{\tilde{g}^{\sigma}}(\lambda,\xi)\\
&+\tilde{g}^{\sigma}(\lambda,\xi)\overline{\tilde{f}^{\sigma}}(\lambda,\xi)\big )\eta(\lambda)\,d\mu_{\sigma}(\xi).
\end{align*}
Then for $ \epsilon<\epsilon_{\max}$  we have
\begin{enumerate}
\item $A(\lambda)$ and $B(\lambda)$  admit a holomorphic extension up to $\operatorname{Im}\lambda=\epsilon$.
\item $A(\lambda)$ and $B(\lambda)$ are even.
\item For every $N\in\N$ there are constants $A_N$ and $B_N$ such that for every $\lambda\in \C$ with $ \lvert\operatorname{Im}\lambda\rvert \leq \epsilon<\epsilon_{\max}$ we have:
\begin{align*}
(i)\,\lvert A(\lambda)&\rvert \leq A_N(\epsilon_{max}-\epsilon)^{-1}(1+\lvert \lambda\rvert)^{-N}e^{2R\epsilon},\\
(ii)\,\lvert \lambda B(\lambda)&\rvert \leq B_N(\epsilon_{max}-\epsilon)^{-1}(1+\lvert \lambda\rvert)^{-N}e^{2R\epsilon}.
\end{align*}
\item We have for $\lvert \operatorname{Im}\lambda \rvert \leq \epsilon$:
\begin{align*}
\frac{4}{C_0}\Big( \mathcal{K}(\varphi)(t)-\mathcal{P}(\varphi)(t)\Big )=\int_{-\infty}^{\infty}\Big(A(\lambda)+i\lambda B(\lambda)\Big)e^{2i\lambda t}\,d\lambda.
\end{align*}
\end{enumerate}
\end{lemma}
\begin{proof}
 (1) is a direct consequent of the the first assertion from Lemma \ref{lemma:lemma} and Corollary \ref{folg:PW}. (3) also follows form  Corollary \ref{folg:PW} by the assumption on the $\mathbf{c}$-function. And if we have that $A$ and $B$ are even then also (4) follows with the same arguments as in Lemma \ref{lemma:Huy1}. Therefore all that remains to show is (2) but this follows immediately from Lemma \ref{lemma:lemma}.
\end{proof}

\begin{proof}[Proof Theorem \ref{thm:energy2}]
With the same argument as in Theorem \ref{thm:Huy} we can restrict ourselves to the case $t\geq 0$.
Let $0<\epsilon<\epsilon_{max}$ then we have by using Lemma \ref{lemma:contur} and shifting the integral to $\R+i\epsilon$:
\begin{align*}
&\Big\lvert \frac{4}{C_0}\Big(\mathcal{K}(\varphi)(t)-\mathcal{P}(\varphi)(t)\Big)\Big\rvert=\Big\lvert \int_{-\infty}^{\infty}\Big (A(\lambda)+i\lambda B(\lambda)\Big)e^{2i\lambda t}\,d\lambda\Big\rvert\\
&=\Big\lvert e^{-2\epsilon t}\int_{-\infty}^{\infty}\Big (A(a+i\epsilon)+i(a+i\epsilon)B(a+i\epsilon)\Big)e^{2iat}\,da\Big\rvert.
\end{align*}
Hence we obtain using the bounds form Lemma \ref{lemma:energy2} that for every $N\in \N$ there is a constant $C_N$ such that for all $\lambda\in \C$ with $\lvert\operatorname{Im} \lambda \rvert \leq\epsilon<\epsilon_{max}$ we have that the above is bounded by 
\begin{align*}
C_N(\epsilon_{max}-\epsilon)^{-1}e^{2R\epsilon}e^{-2\epsilon t}\int_{-\infty}^{\infty}(1+\lvert \lambda \rvert)^{-N}\,d\lambda \quad \forall t\geq 0. 
\end{align*}
And since the integral is bounded we get that there is a constant $C>0$ such that the above is bounded by:
\begin{align*}
C(\epsilon_{max}-\epsilon)^{-1}e^{-2\epsilon(\lvert t\rvert-R)}\quad \forall t\geq 0. 
\end{align*}
For the case that the $\mathbf{c}$-function is an entire function and a polynomial one notice that we can ignore the therm $(\epsilon_{max}-\epsilon)^{-1}$ in all the estimates and then we can let $\epsilon\to \infty$ which yields the the assertion. 
\end{proof}
\begin{bem}
Note that the assumption on the pole of $\eta$ to be of multiplicity one only effects the therm $(\epsilon_{max}-\epsilon)^{-1}$  so one could restate  Theorem \ref{thm:Huy} and Theorem \ref{thm:energy2} for $\eta$ to have a pole of multiplicity $n\in \N$ by raising the power to $-n$. But there are no known examples for this case, even for $\mathbf{c}$-functions on hypergroups. Hence we state our theorems in the realistic setting.
\end{bem}

\footnotesize
\bibliography{literature}
\bibliographystyle{alpha}

\end{document}